\documentclass[final]{myarticle}

\usepackage[T1]{fontenc}
\usepackage{dsfont}
\usepackage{mathrsfs}
\usepackage[colorlinks, citecolor=blue, linkcolor=blue]{hyperref}
\usepackage{xcolor}
\usepackage{mathscinet}
\usepackage{latexsym}
\usepackage{amsthm}
\usepackage{amssymb}
\usepackage{amsfonts}
\usepackage{amsmath}
\usepackage{longtable}
\usepackage{graphicx}
\usepackage{multirow}
\usepackage{multicol}
\usepackage{amsfonts, amsmath}
\usepackage{latexsym,bm,amsfonts,amssymb,pifont,mathbbol,bbm}
\usepackage{extarrows}
\usepackage{verbatim}
\allowdisplaybreaks
\usepackage[a4paper]{geometry}

\DeclareMathAlphabet{\mathpzc}{OT1}{pzc}{m}{it} 

\newtheorem{theorem}{Theorem}[section]
\newtheorem{thm}[theorem]{Theorem}
\newtheorem{lemma}[theorem]{Lemma}
\newtheorem{lem}[theorem]{Lemma}

\newtheorem{proposition}[theorem]{Proposition}
\newtheorem{prop}[theorem]{Proposition}

\theoremstyle{definition}

\theoremstyle{remark}
\newtheorem{remark}[theorem]{Remark}
\newtheorem{rem}[theorem]{Remark}
\numberwithin{equation}{section}

\newcommand{\set}[1]{\left\{#1\right\}}
\newcommand{\im}{\mathrm{i}}
\newcommand{\mi}{\mathrm{i}}
\newcommand{\abs}[1]{\left\vert#1\right\vert}
\newcommand{\dif}{\mathrm{d}}
\newcommand{\C}{\mathbb{C}}
\newcommand{\R}{\mathbb{R}}

\newcommand{\N}{\mathbb{N}}
\newcommand{\norm}[1]{\left\Vert#1\right\Vert}
\newcommand{\FH}{\mathfrak{H}}
\newcommand{\RE}{\mathbb{Re}}
\newcommand{\IM}{\mathbb{Im}}
\newcommand{\E}{\mathbb{E}}
\newcommand{\tensor}[1]{\mathsf{#1}}

\newcommand{\innp}[1]{\langle {#1}\rangle}

\date{}
\begin{document}
	
	\title[complex Berry-Ess\'een bound ]{An improved complex fourth moment theorem}
	%%% \subtitle{Sub-Title of Your Article}%%% optional
	\author{Huiping \textsc{Chen}}% Author Name (\sc should NOT be used here)
	\address{LMAM, School of Mathematical Sciences, Peking University, Beijing, 100871, China}
	\email{chenhp@pku.edu.cn}
	
	\author{Yong \textsc{Chen}}% Author Name (\sc should NOT be used here)
	\address{Corresponding author\\Center of Applied Mathematics, School of Mathematics and Statistics,
		Jiangxi Normal University, Nanchang, Jiangxi 330022, China}
	\email{zhishi@pku.org.cn}
	
	\author{Yong \textsc{Liu}}% Author Name (\sc should NOT be used here)
	\address{LMAM, School of Mathematical Sciences, Peking University, Beijing, 100871, China}
	\email{liuyong@math.pku.edu.cn}
	
	% Subject code(s)
	
	\keywords{Complex Wiener chaos, Berry-Ess\'een bound, Fourth Moment Theorem. \\\indent 	
	We thank Prof. Xiaohong Lan for valuable comments and discussion. Y. Chen is supported by NSFC (No. 11961033). Y. Liu is supported by NSFC (No. 11731009, No. 12231002) and Center for Statistical Science, PKU}% Key word(s)

	\begin{abstract}
		For a series of univariate or multivariate complex multiple Wiener-It\^{o} integrals, we appreciably improve the previously known contractions condition of complex Fourth Moment Theorem (FMT) and present a fourth moment type Berry-Ess\'een bound under Wasserstein distance.
		% Especially, when the complex multiple Wiener-It\^{o} integral is  circularly-symmetric, both the FMT and Berry-Ess\'een bound are much more concise.
		Note that in some special cases of univariate complex multiple Wiener-It\^{o} integral, the Berry-Ess\'een bound we acquired is optimal. A remarkable fact is that the Berry-Ess\'een bound of multivariate complex multiple Wiener-It\^{o} integral is related to the partially order of the index of the complex multiple Wiener-It\^{o} integral, which has no real counterparts as far as we know.  As an application, we explore the asymptotic property for the numerator of a ratio process which originates from the classical Chandler wobble model. 
	\end{abstract}
	
	\maketitle	
	\section{Introduction}
	
	In 1952, It\^o published the groundbreaking article \cite{ito} and established the theory of complex multiple Wiener-It\^o integrals (also called complex Wiener chaos random variables) with respect to a complex normal random measure. Since then, there has been renewed interest in theoretical research on complex Wiener-It\^o integrals. For instance, Hida concentrated on the theory of complex multiple Wiener integrals under the framework of complex white noise and from an analytical perspective in \cite[Chapter 6]{HD80}. In \cite{Camp15,chw17,cl17}, the authors proved the Fourth Moment Theorem (FMT) for complex multiple Wiener-It\^o integrals, which states that the convergence in law of a sequence of complex multiple Wiener-It\^o integrals to a complex Gaussian distribution is equivalent to the convergences of the absolute moment up to fourth. Additionally, the product formula and \"Ust\"unel-Zakai independence criterion for complex multiple Wiener-It\^o integrals were acquired in \cite{ch17}. In \cite{cl19}, the authors explored various properties of complex multiple Wiener-It\^o integrals, complex Ornstein-Uhlenbeck operators and semigroups. 
	
	The study of complex Gaussian fields is highly motivated by various applications. For instance, it is crucial to understand the asymptotic behavior of some functionals of complex  Wiener-It\^o integrals in probabilistic models of cosmic microwave background radiation, see \cite{Kamionkowski_1997,marinucci_peccati_2011}. The theory of complex Wiener-It\^o integrals can be used to analyze the stochastic cubic complex Ginzburg-Landau equation, one of the most significant nonlinear partial differential equations in applied mathematics and physics, see \cite{Hoshino2017,RevModPhys.74.99}. The Ornstein-Uhlenbeck process, introduced in \cite{arato1982linear,arato1962evaluation} to model the Chandler wobble or variation of latitude concerning the rotation of the earth, has also found widespread use in finance and econophysics. Techniques involving complex Wiener-It\^o integrals can be used to obtain statistical inference for parameter estimators of the Ornstein-Uhlenbeck process, such as consistency and asymptotic normality, see \cite{chw17,sty22} for example. In the field of communication and signal processing, noise is frequently assumed to be complex Gaussian noise, see \cite{Aghaei2008,BARONE2005,Matalkah2008,Reisenfeld2003} for more details.

	In this paper, we investigate the FMT and  Berry-Ess\'een bound for complex Wiener chaos random variables and vectors. We first improve the FMT presented in the literature, see \cite[Theorem 1.1]{chw17} and \cite[Theorem 3.14]{cl19}. Specifically, in Theorem \ref{equilent cond}, we removed some redundant contraction conditions in the equivalent conditions of FMT. One always calculates contraction conditions to verify whether FMT is valid. Therefore, this improvement has practical significance in the application of FMT to derive the asymptotic normality for a sequence of complex Wiener-It\^o integrals. Moreover, in Theorem \ref{cFMBE bound}, we obtain the Berry-Ess\'een lower and upper bounds for a complex Wiener-It\^o integral and a bivariate Gaussian random vector with the same covariance. The fourth moment type Berry-Ess\'een upper bound we yielded simplifies that in \cite[Theorem 4.6]{Camp15}. Note that in some special cases, the Berry-Ess\'een bound we acquired is optimal, see Theorem \ref{example_BM} as an example. Compared to \cite[Theorem 4.1]{c23}, in which the author obtained the optimal bound under some smooth enough distance, we derive the bound under Wasserstein distance. 
	
	In Section \ref{section 3.2}, we focus on vector-valued complex Wiener-It\^o integrals. Here, we make some assumptions of circular symmetry. From the viewpoint of mathematics, circularly-symmetric complex random vectors have graceful properties and are convenient for calculations and analyses, see Section \ref{section 2.2} \cite[Section 3.7]{GRG13} and \cite[Section 3.6, 3.7]{ito2001} for more details. And from the viewpoint of applications, it turns out to be suited to model random fields arising in cosmology, see \cite{bmv07,Kamionkowski_1997, mp08,marinucci_peccati_2011}. In Theorem \ref{thm Camp}, we get the multivariate complex fourth moment theorem which shows that for a sequence of vector-valued circularly-symmetric complex Wiener-It\^o integrals, componentwise convergence to Gaussian always implies joint convergence. As a corollary, in Theorem \ref{thm Pecca TUd}, we establish some sufficient conditions for the asymptotic normality of a complex square integrable random variable of which chaos expansion satisfies some circularly symmetrical conditions. Finally, in Theorem \ref{duowei b-e bound thm}, the Berry-Ess\'{e}en upper bound for the multivariate complex fourth moment is acquired. Note that this Berry-Ess\'{e}en upper bound is related to a partially order relation of the indices of the components of a Wiener chaos random vector and may improve the corresponding estimate in the real case  (see \cite[Theorem 6.2.2]{NourPecc12}). See Remark \ref{remark3.12} for the further explanation.

	%We will emphasize that there exist some essential differences between the complex chaos random variables and the real chaos random variables.  
	
	It is shown in \cite[Theorem 3.3]{cl17} that a complex $(p, q)$-th Wiener chaos random variable is equivalent to a two-dimensional random vector whose components are real $(p+q)$-th Wiener chaos random variables associated with a direct sum Hilbert space, which implies that both kernels of the components of real Wiener chaos random vector are $2^{p+q}$-dimensional vector-valued functions. Recently, we also find an explicit formula to represent these two $2^{p+q}$-dimensional functions, see \cite{ccl22}. That is to say, for any given complex $(p, q)$-th Wiener chaos random variable, the dimensions of kernels of the complex chaos variable and the equivalent two-dimensional real chaos vector are $1$ and $2^{p+q}$ respectively. This big growth of dimensions is a huge cost when we try to use this represent formula and hence we call it a dimension problem. 
	
	This dimension problem reminds us that there exist essential differences between the complex Wiener chaos random variables and the real Wiener chaos random vectors. Hence, it is more realistic to further develop the theories of complex Wiener chaos random variables other than to use the tool of real Wiener chaos random vectors with very high dimensional kernels. Moreover, we are more interested to seek the properties of complex Wiener chaos random variables that have no real counterparts. In fact, there really exist these types of properties.   
	
	In the  the following two examples, this type of property appears naturally. Recall firstly the complex Ornstein-Uhlenbeck process which models the Chandler wobble:
	\begin{equation}\label{cp}
		\dif Z_t=-\gamma Z_t\dif t+ \sqrt{a}\dif \zeta_t\,,\quad t\ge 0,\,\, Z_0=0,
	\end{equation}%
	where $Z_t=X_1(t)+\mi X_2(t)$ is a complex-valued process, $\gamma=\lambda-\mi \omega,\,\lambda>0$, $a>0$ and $\zeta_t
	$ is  a fractional Brownian motion with Hurst parameter $H\in [\frac12, \frac34)$. 
	%It was first proposed by M. Arat\'{o} under the directions of A.N. Kolmogorov and Ya.G. Sinai \cite{arato1962evaluation} and 
	For the special case $H=\frac12$, the maximum likelihood estimate for the parameter $\gamma$ was studied, please refer to \cite{arato1982linear, arato1962evaluation, LS01} for details.  
	For the case of $H\in (\frac12,\frac34)$, the authors in \cite{chw17} investigated essentially the asymptotic property of a ratio process as follows
	\begin{equation}\label{hat gamma}
		\frac{\frac{1}{\sqrt{T}}\int_0^T \bar{Z}_t\dif \zeta_t}{\frac{1}{T}\int_0^T \abs{Z_t}^2\dif t},
	\end{equation}
	%--------------------------------------------------------------------------------------------------------
	where as usual, $\bar{Z}_t,\,\abs{Z_t}$ are respectively the complex conjugate and the modulus of $Z_t$. Formally, the above ratio process can be looked as the bias of the so-called least squares estimator $\hat{\gamma}$ of the parameter $\gamma$ although $\hat{\gamma}$ is not a real estimator from the viewpoint of statistics.  We can rewrite the numerator as a complex chaos process for $T\ge 0$:
	\begin{align}\label{FT dyi}
		F_T:= {I}_{1,1}(\psi_{T}(t,s)),
	\end{align}where 
	\begin{align}\label{buduichen kenel}
		\psi_{T}(t,s)=\frac{1}{\sqrt{T}} e^{-\bar{\gamma}(t-s)}\mathbf{1}_{\left\lbrace 0\leq s\leq t\leq T\right\rbrace},
	\end{align}
	and $\mathbf{1}_{E}$ is the indicator function of a set $E$. Moreover, we apply the product formula of complex multiple integrals and stochastic Fubini theorem to rewrite the denominator as follows.
	\begin{align*}
		\frac{1}{T} \int_0^T \abs{Z_t}^2\mathrm{d} t =  \frac{1}{2\lambda T}I_{1,1}(f_T(t,s))+ \frac{1}{T} \int_0^T\,  E\abs{Z_t}^2 \dif t
	\end{align*}where 
	\begin{align}\label{duichen kenel}
		f_T(t,s)=e^{-\bar{\gamma}(t-s)}\mathbf{1}_{\left\lbrace 0\leq s\leq t\leq T\right\rbrace}+e^{-{\gamma}(s-t)}\mathbf{1}_{\left\lbrace 0\leq t\leq s\leq T\right\rbrace} -e^{-\gamma (T-t)-\bar{\gamma} (T-s)}\mathbf{1}_{\set{0\le s,t\le T}}.
	\end{align}
	Hence, \begin{align}
		\frac{1}{T} \int_0^T \abs{Z_t}^2\mathrm{d} t = \frac{1}{2\lambda }\Big[\frac{ F_T+\bar{F}_T}{\sqrt{T}}  - \frac{\abs{Z_T}^2-E\abs{Z_T}^2}{T}\Big] + \frac{1}{T} \int_0^T\,  E\abs{Z_t}^2 \dif t.
		\label{x square}
	\end{align}
	
	To compare the kernel in \eqref{buduichen kenel}  with that in \eqref{duichen kenel}, we find that the former is completely asymmetric, but the latter is Hermite symmetric. This type of asymmetry has no real counterparts. It is known that  one of the most effective ways of dealing with real second chaos random variables is two objects: the Hilbert-Schmidt operator and the sequence of auxiliary kernels, see \cite[Section 2.7.4]{NourPecc12}. But for asymmetric kernels of complex second Wiener chaos random variables such as \eqref{buduichen kenel}, the two objects can not be applied any more. This is a main reason why it is more difficult to deal with complex  second Wiener chaos random variables than the real one. 
	
	Secondly,  as an extension of the model \eqref{cp}, the complex Vasicek model defined as
	\begin{align}\label{complex Vasicek}
		\dif Z_t=\gamma (b-Z_t)\dif t+ \sqrt{a}\dif \zeta_t\,,\quad t\ge 0,\,\, Z_0=0,
	\end{align} where $\gamma$ is the same as in \eqref{cp} and $b$ is a complex number.
	The asymptotic properties of the estimators of the parameters $\gamma$ and $b$ are related to two much more complicated ratio processes than \eqref{hat gamma}, where $3$-th complex Wiener chaos random variables with asymmetric kernels are involved, please refer to \cite{sty22} for details. From the above two examples, the theory of complex Wiener chaos random variables
	is necessary to be further developed and completed.

	The paper is organized as follows. Section \ref{section 2} introduces some elements of complex multiple Wiener-It\^o integrals and circularly-symmetric complex random vectors. In Section \ref{section 3.1} and Section \ref{section 3.2}, we obtain the simplified complex FMT and Berry-Ess\'een bound for a $d$-dimensional random vector of which components are complex Wiener-It\^o ingegrals, where $d\geq1$. For As an application, in section \ref{section 3.3}, we derive the optimal Berry-Ess\'een bound and Berry-Ess\'een upper bound for a statistic associated with complex Ornstein-Uhlenbeck process and complex fractional Ornstein-Uhlenbeck process, respectively. In Section \ref{section 4}, some technical estimates are shown and then the main results in Section \ref{section 3} are proved.
	
	\section{Preliminaries}\label{section 2}
	\subsection{Complex multiple Wiener-It\^{o} integral}\label{section 2.1}
	Let us now fix the mathematical framework of complex Wiener chaos random variables. See \cite{cl17, HD80, janson, ito} for more details.
	%\section{Preliminary:   Complex Wiener chaos and the product formula }\label{comp w-i chaos}
	Suppose that $\FH$ is a complex separable Hilbert space. Recall that a complex Gaussian isonormal process $\set{Z(h):\,h\in \FH}$ is  a centered symmetric complex Gaussian family in $L^2(\Omega)$ such that
	\begin{align*}
		\mathbb{E}[ Z({h})^2]=0,\quad \mathbb{E}[Z({g})\overline{Z({h})}]=\innp{{g}, {h}}_{\FH},\quad \forall {g},{h}\in \FH.
	\end{align*}  
	For each $p,q\ge 0$, we write $\mathscr{H}_{p,q}(Z)$ to indicate the closed linear subspace of $L^2(\Omega)$ generated by the random variables of the type
	$$\set{H_{p,q}(Z( {h})):\, {h}\in \mathfrak{H},\,\norm{ {h}}_{\mathfrak{H} }=\sqrt2}$$ where $H_{p,q}(z)$ is the complex Hermite polynomial, or Hermite-Laguerre-It\^o polynomial, given by 
	\begin{equation*}
		\exp\left\{ \lambda\bar{z}+\bar{\lambda}z-2|\lambda|^2\right\} =\sum_{p=0}^{\infty}\sum_{q=0}^{\infty}\frac{\bar{\lambda}^p\lambda^q}{p!q!}H_{p,q}(z),\;\lambda\in\mathbb{C}.
	\end{equation*}
	The space $\mathscr{H}_{p,q}(Z)$ is called the Wiener-It\^{o} chaos of degree of $(p,q)$ of $Z$ (or say: $(m,n)$-th Wiener-It\^{o} chaos of $Z$).
	%-------------------------------------------------------------------------------------------------------
	Take a complete orthonormal system $\set{ {e}_j,\,j\ge 1}$ in $\mathfrak{H} $. For two sequences $$\mathbf{p}=\set{p_j}_{j=1}^{\infty},\,\mathbf{q}=\set{q_j}_{j=1}^{\infty}$$ of non-negative integers with only finitely many nonzero components, define a Fourier-Hermite polynomial
	\begin{equation*}\label{fourier}
		\mathbf{J}_{\mathbf{p},\mathbf{q}}:=\prod_{j} \frac{1}{\sqrt{2^{p_j+q_j}p_j!q_j!}}H_{p_j,q_j}(\sqrt2 Z({e}_j)).
	\end{equation*}
	Then, for any $p,q\ge0$, the random variables
	\begin{equation*}\label{eq4}
		\set{\mathbf{J}_{\mathbf{p},\mathbf{q}}:\abs{\mathbf{p}} =p,\,\abs{\mathbf{q}}=q }
	\end{equation*}
	form a complete orthonormal system in $\mathscr{H}_{p,q}(Z)$. The linear mapping
	%-------------------------------------------------------------------------------------------------------
	\begin{equation*}\label{mulint}
		{I}_{p,q}\big((\hat{\otimes}_{j=1}^{\infty} {e}_j^{\otimes p_j})\otimes (\hat{\otimes}_{j=1}^{\infty}\bar{ {e}}_j^{\otimes q_j})\big)=\sqrt{\mathbf{p}!\mathbf{q}!}\mathbf{J}_{\mathbf{p},\mathbf{q}}
	\end{equation*} 
	%-------------------------------------------------------------------------------------------------------
	provides an isometry from the tensor product $\FH^{\odot p}\otimes \FH^{\odot q}$, equipped with the norm $\sqrt{p!q!}\norm{\cdot}_{\FH^{\otimes (p+q)}}$, onto the $(p,q)$-th Wiener-It\^{o} chaos $\mathscr{H}_{p,q}(Z)$ (see \cite[Theorem 13.2]{ito}). Here, $e_i\hat{\otimes} e_j$ means that the symmetrizing tensor product of $e_i$ and $e_j$ and $\FH^{\odot q}$ is the $q$ times symmetric tensor  product of $\FH$. %Then one can define complex multiple Wiener-It\^{o} integrals as follows.
	For any $f\in \FH^{\odot p}\otimes \FH^{\odot q}$, we call $ {I}_{p,q}(f)$ the complex multiple Wiener-It\^{o} integral of $f$ with respect to $Z$.
	For any $f\in \mathfrak{H}^{\otimes (p+q)}$, %given by $$f= \sum_{j_1,\dots,j_{p+q}}a_{j_1,\dots,j_{p+q}}e_{j_1}\otimes\dots e_{j_p}\otimes \bar{e}_{j_{p+1}}\otimes\dots\otimes \bar{e}_{j_{p+q}},$$
	define $$I_{p,q}(f)=I_{p,q}(\tilde{f}),$$ where $\tilde{f}$ is the symmetrization of $f$ in the sense of \cite[Equation 5.1]{ito} by It\^{o}. In this article, $F={I}_{p,q}(f)$ with $f\in \FH^{\odot p}\otimes \FH^{\odot q}$ is also called a complex $(p,q)$-th Wiener chaos random variable.
	%\begin{equation*}\label{ito-sense}  \tilde{f}=\frac{1}{p!q!}\sum_{\pi}\sum_{\sigma}\sum_{j_1,\dots,j_{p+q}}a_{j_1,\dots,j_{p+q}}e_{j_{\pi(1)}}\otimes\dots\otimes e_{j_{\pi(p)}}\otimes \bar{e}_{j_{\sigma(1)}}\otimes\dots\otimes \bar{e}_{j_{\sigma(q)}},  \end{equation*} where $\pi$ and $\sigma$ run over all permutations of $(1,\dots,p)$ and $(p+1,\dots, p+q)$ respectively.
	Complex Wiener-It\^o chaos decomposition of $L^2(\Omega)$ implies that $L^2(\Omega)$ can be decomposed into the infinite orthogonal sum of the spaces $\mathscr{H}_{p,q}(Z)$. That is, any random variable $F \in L^2(\Omega)$ admits a unique expansion of the form
	\begin{equation}\label{complex chaos decomposition}
		F=\sum_{p=0}^{\infty}\sum_{q=0}^{\infty} I_{p,q}\left(f_{p,q}\right),	
	\end{equation}
	where $f_{0,0}=\E[F]$, and $f_{p,q} \in \mathfrak{H}^{\odot p}\otimes\mathfrak{H}^{\odot q}$ with $p+q \geq1$, are uniquely determined by $F$.
	
	Given $f\in\mathfrak{H}^{\odot a}\otimes\mathfrak{H}^{\odot b}$, $g\in\mathfrak{H}^{\odot c}\otimes\mathfrak{H}^{\odot d}$, for $i=0,\dots,a\land d$, $j=0,\dots,b\land c$, the $(i,j)$-th contraction of $f$ and $g$ is an element of $\mathfrak{H}^{\odot (a+c-i-j)}\otimes\mathfrak{H}^{\odot (b+d-i-j)}$ defined by
	\begin{align*}
		f \otimes_{i, j} g
		&= \sum_{l_{1}, \ldots, l_{i+j}=1}^{\infty}\left\langle f, e_{l_{1}} \otimes \cdots \otimes e_{l_{i}} \otimes \bar{e}_{l_{i+1}} \otimes \cdots \otimes \bar{e}_{l_{i+j}}\right\rangle\\ &\qquad\qquad\qquad\otimes\left\langle g, e_{l_{i+1}} \otimes \cdots \otimes e_{l_{i+j}} \otimes \bar{e}_{l_{1}} \otimes \cdots \otimes \bar{e}_{l_{i}}\right\rangle,
	\end{align*}
	and by convention, $f \otimes_{0,0} g=f \otimes g$ denotes the tensor product of $f$ and $g$. \cite[Theorem 2.1]{ch17} and  \cite[Theorem A.1]{Hoshino2017} establish the product formula for complex Wiener-It\^o integrals. For $f \in \mathfrak{H}^{\odot a} \otimes\mathfrak{H}^{\odot b}$ and $ g \in \mathfrak{H}^{\odot c} \otimes \mathfrak{H}^{\odot d}$ with $a, b, c, d \geq0$,
	\begin{equation}\label{Product_formula}
		I_{a, b}(f) I_{c, d}(g)=\sum_{i=0}^{a \wedge d} \sum_{j=0}^{b \wedge c}\binom{a}{i} \binom{d}{i}\binom{b}{j}\binom{c}{j} i! j! I_{a+c-i-j, b+d-i-j}\left(f \otimes_{i, j} g\right).
	\end{equation}
	
	Suppose that  $  f\in   \FH^{\odot p}\otimes \FH^{\odot q}$ and denote $\bar{f}$ the complex conjugate of $  f $. Then we call a function  $h\in \FH^{\odot q}\otimes \FH^{\odot p}$, given as follows,
	$$  h(t_1,\dots,t_q,s_1,\dots,s_p):=\bar{f}(s_1,\dots,s_p,t_1,\dots,t_q),$$ the reverse complex conjugate of $  f $.
	It is well known that 
	\begin{align*}
		\overline{{I}_{p,q}(f)}={I}_{q,p}(h),
	\end{align*}where $\overline{{I}_{p,q}(f)}$ is  the complex conjugate of ${I}_{p,q}(f)$.
	
	It\^o in \cite[Theorem 7]{ito} showed that complex Wiener -It\^o integrals satisfy the isometry property. That is, for $f\in\mathfrak{H}^{\odot a}\otimes\mathfrak{H}^{\odot b}$ and $g\in\mathfrak{H}^{\odot c}\otimes\mathfrak{H}^{\odot d}$,
	\begin{equation}\label{isometry property}
		\E\left[ I_{a,b}(f)\overline{{I}_{c,d}(g)}\right]=\mathbf{1}_{\left\lbrace a=c \right\rbrace } \mathbf{1}_{\left\lbrace b=d \right\rbrace } a!b!\left\langle f,g\right\rangle _{\FH^{\otimes (a+b)}}.
	\end{equation}

	%-------------------------------------------------------------------------------------------------------------
	Denote by $\FH_{\R}$ the real Hilbert space such that $\FH=\FH_{\R}+\mi \FH_{\R}$ and a realization of the isonormal Gaussian process $W$ over the Hilbert space direct sum space $\FH_{\R}\oplus \FH_{\R}$ is
	\begin{align*}              W(h,f)&=X(h)+Y(f),\qquad\forall h,f\in\FH_{\R}.         \end{align*}
	Suppose $\varphi\in \FH^{\odot p}\otimes \FH^{\odot q}$ and $F={I}_{p,q}(\varphi)$. Then 
	there exist $u,\,v\in (\FH_{\R}\oplus \FH_{\R})^{\odot (p+q)}$ such that
	\begin{align}
		F=\mathcal{I}_{p+q}(u)+\mi\, \mathcal{I}_{p+q}(v),\label{pp2 Connection}
	\end{align}
	where $\mathcal{I}_p(g)$ is the $p$-th real Wiener-It\^{o} multiple integral of $g$ with respect to $W$ (see \cite[Theorem 3.3]{cl17}). Please refer to \cite{ccl22} for the explicit expressions of the kernels $u$ and $v$. Hence, in this article, we can also call $F$ a $(p+q)$-th complex Wiener chaos without ambiguity.
	
	\subsection{Circularly-symmetric complex random vector}\label{section 2.2}
	In this section, we introduce the definition and some properties of circularly-symmetric complex random vectors, see \cite[Section 3.7]{GRG13} for more details. Throughout the paper, the distribution of a $d$-dimensional complex-valued random vector $Z$ means the 
	joint distribution of 
	$2d$-dimensional real-valued random vector composed by $2d$ real and imaginary components of $Z$.
	
	We say that a complex-valued random variable $Z$ is circularly symmetrical if the distribution of $Z$
	remains invariant under the transformation $z\to e^{\mi \alpha} z$ for any $\alpha\in [0,2\pi)$. It is clear that the circular symmetry of $Z$ implies that the pseudocovariance of $Z$ vanishes, i.e., $\E[Z^2]=0$. As usual, we denote by $ \mathcal{N}_d(0,\Sigma)$ a $d$-dimensional real-valued Gaussian random vector with covariance matrix $\Sigma$. We omit the subscript if $d=1$. We say that a complex-valued random variable $Z$ is Gaussian if the real and imaginary parts of $Z$ are jointly Gaussian. If $Z$ has a circularly-symmetric univariate complex Gaussian distribution with mean zero and variance $\sigma^2=\E[\abs{Z}^2]$, 
	we denote by $$Z\sim \mathcal{CN}(0,\sigma^2),$$ where the $\mathcal{C}$ denotes that $Z$ is both circularly-symmetric and complex. Note that the real and imaginary parts of $Z$ are then real independent identically distributed Gaussian random variables with variance $\frac{\sigma^2}{2}$ each. 
	
	Next, we introduce circularly-symmetric complex random vector. The definition of circularly-symmetric complex random variables can be extended to $d$-dimensional complex random vectors as follows.   %-------------------------------------------------------------------------------------------------------
	Let $Z=(Z_1,Z_2,\dots, Z_d)'$ be a $d$-dimensional complex-valued random vector, then we say that $Z$ is circularly symmetrical if the distribution of $Z$
	remains invariant under the transformation $z\to e^{\mi \alpha} z$ for any $\alpha\in [0,2\pi)$, i.e., \begin{align*}
		Z\stackrel{ {d}}{=} e^{\mi \alpha} Z
	\end{align*} holds for any $\alpha\in [0,2\pi)$,
	where $\stackrel{ {d}}{=}$ means that two random variables have the same distribution. The circular symmetry of the random vector $Z$ also implies that the pseudocovariance of $Z$ vanishes, i.e., $\E[Z {Z}']=0$.
	We say that $Z$ is Gaussian if the $2d$ real and imaginary components of $Z$ are jointly Gaussian.
	We denote by $$Z\sim \mathcal{CN}_d(0,\Sigma)$$ if $Z$ has a circularly symmetrical multivariate complex Gaussian distribution with mean zero and covariance $\Sigma=\E[Z\bar{Z}']$.

	\section{Main results}\label{section 3}
	\subsection{Univariate complex FMT and Berry-Ess\'{e}en bound}\label{section 3.1}
	
	The following version of complex FMT improves the previous results in \cite[Theorem 1.3]{chw17} and \cite[Theorem 3.14]{cl19}, i.e., the redundant contraction conditions are removed. 
	%\section{complex FMT,  complex multivariate FMT and complex chaotic central limit theorem}
	\begin {theorem}[FMT]\label{equilent cond}
	Suppose $p$ and $q$ are fixed natural number such that $p+q\ge 2$.  Let   $\set{F_{n}=I_{p,q}(f_n)}_{n\geq1}$ be a sequence of $(p,q)$-th complex Wiener chaos random variable. If $\E[\abs{F_n}^2]\to \sigma^2$ and $\E[F_n^2]\to c+\mi b $ as $n\to \infty$, then the following statements  are equivalent:
	\begin{itemize}
		\item[\textup{(i)}]The sequence $(\RE F_n,\,\IM  F_n)$ converges in law to a bivariate normal distribution with covariance matrix $\tensor{C}= \frac{1}{2}\begin{bmatrix}\sigma^2+c & b\\ b & \sigma^2-c  \end{bmatrix}$ as $n\rightarrow \infty$.
		\item[\textup{(ii)}]  $\norm{f_n{\otimes}_{i,j} h_n}_{\FH^{\otimes ( 2(l-i-j))}}\to 0 $ as $n\rightarrow \infty$ for any $0<i+j\le p+q-1$, where $h_n$ is the kernel of $\bar{F}_n$ (the complex conjugate of $F_n$), i.e., $\bar{F}_n=I_{q,p}(h_n)$.
	\end{itemize}
	%where 
	%--------------------------------------------------------------------------------------------------------
	\end {theorem} 
	
	\begin{remark}
		Inspired by \cite[Proposition 5]{mp08}, we establish Theorem \ref{equilent cond}. Note that the definition of complex chaos random variable in \cite{mp08} is different from that in our framework.
	\end{remark}
	\begin{remark} 
		It is clear that Theorem \ref{equilent cond} appreciably simplify the  previous results in \cite[Theorem 1.3]{chw17} and \cite[Theorem 3.14]{cl19} which states that (i) is equivalent to that 
		\begin{itemize}
			\item[\textup{(ii')}]  $\norm{f_n\otimes_{i,j} f_n}_{\FH^{\otimes ( 2(l-i-j))}}\to 0$ for any $0<i+j\le l'-1$, where $l'=2(p\wedge q)$, and $\norm{f_n{\otimes}_{i,j} h_n}_{\FH^{\otimes ( 2(l-i-j))}}\to 0 $ for any $0<i+j\le p+q-1$.
		\end{itemize}
		This improvement is significant and very useful when applying complex FMT to get the asymptotic normality of $\set{F_{n}=I_{p,q}(f_n)}_{n\geq1}$. For example, we now can delete almost all the proof of \cite[Lemma 3.6]{chw17} to obtain the asymptotic normality of the complex random variable $F_T$ given as in \eqref{FT dyi}.
	\end{remark}
	\begin{remark}
		According to \cite[Theorem 1.3]{chw17} or \cite[Theorem 3.14]{cl19}, we know that (i) and (ii) are also equivalent to that $$\E[\abs{F_n}^4]-2 \big(\E[\abs{F_n}^2]\big)^2- \big| \E[F_n^2] \big|^2\rightarrow 0,\quad n\rightarrow \infty.$$
		See \cite[Proposition 3.12]{cl19} or \eqref{revised version1}, \eqref{revised version2} for different expressions of the quantity $\E[\abs{F_n}^4]-2 \big(\E[\abs{F_n}^2]\big)^2- \big| \E[F_n^2] \big|^2$. Compared to \cite[Equation 5.2.5, Equation 5.2.6]{NourPecc12}, the expression of fourth cumulant of real Wiener-It\^o integral, \cite[Proposition 3.12]{cl19} is more complicated and associated with more contractions. Due to this reason, unlike Fourth Moment Theorem for real Wiener-It\^o integral (see \cite[Theorem 5.2.7]{NourPecc12}), it is not sure that whether (i) and (ii) are equivalent to that
		$$\norm{f_n\tilde{{\otimes}}_{i,j} h_n}_{\FH^{\otimes ( 2(l-i-j))}}\to 0, $$ for any $0<i+j\le p+q-1$ unless the estimate in Lemma \ref{norm inequality lemma} is still hold for $f_n\tilde{{\otimes}}_{i,j} f_n$. This is one of the reflections that the theory of complex Wiener-It\^o integral is more complicated than the real one. However, one can ignore this problem since one always calculates $\norm{f_n\otimes_{i,j} h_n}_{\FH^{\otimes ( 2(l-i-j))}}$  rather than more complicated quantity $\norm{f_n\tilde{{\otimes}}_{i,j} h_n}_{\FH^{\otimes ( 2(l-i-j))}}$ to yield the asymptotic normality of $\set{F_{n}=I_{p,q}(f_n)}_{n\geq1}$ by utilizing Theorem \ref{equilent cond}.  
	\end{remark}

	%To obtain the Stein-Malliavin approximations, we need the following version of the fourth moment Berry-Ess\'{e}en bound.
	\begin{theorem}\label{cFMBE bound}% [Fourth Moment Berry-Ess\'{e}en bound]
		Suppose that $F$ is a $(p,q)$-th complex Wiener chaos random variable with 
		$p+q\ge 2$ satisfying $\E[\abs{F}^2]= \sigma^2$ and $\E[F^2]= a+\mi b $, and that $N$ has a bivariate normal distribution with mean zero and covariance matrix $$\tensor{C}= \frac{1}{2}\begin{bmatrix}\sigma^2+a & b\\ b & \sigma^2-a  \end{bmatrix}.$$ Then the Wasserstein distance between the laws of $F$ and $N$ satisfies the following fourth moment Berry-Ess\'{e}en bound:
		\begin{align}
			d_{W}(F, N)&\leq 4\sqrt2  \sqrt{ \sum_{r=1}^{l-1} {2r \choose r}}  \frac{ \sqrt{\lambda_1}}{\lambda_2} \sqrt{\E[\abs{F}^4]-2 \big(\E[\abs{F}^2]\big)^2- \big| \E[F^2] \big|^2}\label{Fourth moment BEB }\\
			&\leq c_2(p,q,a,b,\sigma) \sqrt{ \sum_{0<i+j<l}\norm{f {\otimes}_{i,j}h}^2_{\FH^{\otimes(2(l-i-j))}}},\label{bound 2}
		\end{align}
		and 
		\begin{equation}\label{lower bound}
		\begin{aligned}
			d_{W}(F, N)&\geq c_1(p,q,a,b,\sigma)\max\left\lbrace \mathbf{1}_{\left\lbrace p=q \right\rbrace }\left| \sum_{i=0}^{p} i!(p-i)!(p!)^2\binom{p}{i}^4\left\langle f\tilde{\otimes}_{i,p-i}f,h \right\rangle _{\FH^{\otimes (2p)}}\right| , \right.  \\&\left.  \qquad\qquad\qquad\qquad\qquad\mathbf{1}_{\left\lbrace p=q \right\rbrace }\left| \sum_{i=0}^{p} i!(p-i)!(p!)^2\binom{p}{i}^4\left\langle f\tilde{\otimes}_{i,p-i}f,f \right\rangle _{\FH^{\otimes (2p)}}\right| ,\right.  \\&\left.\qquad\qquad\qquad\qquad\qquad\sum_{0<i+j<l}\norm{f {\otimes}_{i,j}h}^2_{\FH^{\otimes(2(l-i-j))}}\right\rbrace .
		\end{aligned}
	\end{equation}
		where $l=p+q$, $c_1(p,q,a,b,\sigma)$and $c_2(p,q,a,b,\sigma)$ are two constants depending on $p,q,a,b,\sigma$, $h$ is the reverse complex conjugate of $f$ and $$\lambda_1=\frac12[\sigma^2+\sqrt{a^2+b^2}],\,\lambda_2=\frac12[\sigma^2-\sqrt{a^2+b^2}]$$ are the two eigenvalues of the matrix $\tensor{C}$.
	\end{theorem}
	
	\begin{remark}
		In Theorem \ref{cFMBE bound}, we get the Berry-Ess\'een lower  bound and upper bound by combining \cite[Theorem 6.2.2]{NourPecc12} (Berry-Ess\'een bound for vector-valued real Wiener-It\^o integral), \cite[Theorem 4.1]{c23} (optimal bound for complex Wiener chaos random variable under some smooth distance) and some estimates for the moments of $F$. Assume that the condition (ii) of Theorem \ref{equilent cond} is valid. Note that if $p\neq q$, then the lower bound is simplified to $$\sum_{0<i+j<l}\norm{f {\otimes}_{i,j}h}^2_{\FH^{\otimes(2(l-i-j))}},$$ which is strictly smaller than the upper bound \eqref{bound 2}. If $p = q$, then $$\sqrt{ \sum_{0<i+j<l}\norm{f {\otimes}_{i,j}h}^2_{\FH^{\otimes(2(l-i-j))}}}$$ will provide an optimal bound in the meaning of that the lower and upper bounds are consistent after neglecting some constant, if it holds that 
		\begin{align*}
			\left| \sum_{i=0}^{p} i!(p-i)!(p!)^2\binom{p}{i}^4\left\langle f\tilde{\otimes}_{i,p-i}f,h \right\rangle _{\FH^{\otimes (2p)}}\right|&\asymp \sqrt{ \sum_{0<i+j<l}\norm{f {\otimes}_{i,j}h}^2_{\FH^{\otimes(2(l-i-j))}}},\\
			\left| \sum_{i=0}^{p} i!(p-i)!(p!)^2\binom{p}{i}^4\left\langle f\tilde{\otimes}_{i,p-i}f,f \right\rangle _{\FH^{\otimes (2p)}}\right|&\asymp \sqrt{ \sum_{0<i+j<l}\norm{f {\otimes}_{i,j}h}^2_{\FH^{\otimes(2(l-i-j))}}}.
		\end{align*}
		Here, for two numerical sequences $\left\lbrace a_n:n\geq1 \right\rbrace $ and $\left\lbrace b_n:n\geq1 \right\rbrace $, we write $a_n\asymp b_n$ if there exist two constants $0<c_1<c_2<\infty$ not depending on $n$ such that $c_1b_n\leq a_n\leq c_2b_n$ for $n$ sufficiently large. 
		See Theorem \ref{example_BM} for an example of this case.
	\end{remark}
	
	\begin{remark}
		If the complex Wiener chaos $F $ is circularly symmetrical (which implies that $c+\mi b=0$) and  $N\sim \mathcal{CN}(0,\,\sigma^2)$, then the above Fourth Moment Berry-Ess\'{e}en bound degenerates to a simple form as
		\begin{align}
			d_{W}(F, N)\leq    \frac{8 }{\sigma} \sqrt{\sum_{r=1}^{l-1} {2r \choose r}}  \sqrt{\E[\abs{F}^4]-2 \big(\E[\abs{F}^2]\big)^2}\label{Fourth moment BEB2} 
		\end{align}
		The above fourth moment type Berry-Ess\'{e}en bounds \eqref{Fourth moment BEB2} appreciably simplify the following bound given in \cite[Theorem 4.6]{Camp15} and its proof:
		\begin{align*}
			% d_{W}(F, N)\le \frac{1 }{\sigma} \sqrt{2 \Psi (F)+ \sqrt{\E(\abs{F}^4)\Psi(F) } },
			d_{W}(F, N)\le \frac{\sqrt2 }{\sigma} \sqrt{ \E[\abs{F}^4]-2 \big(\E[\abs{F}^2]\big)^2 + \sqrt{\frac12\E[\abs{F}^4]\Big(\E[\abs{F}^4]-2 \big(\E[\abs{F}^2]\big)^2 } \Big)}.
		\end{align*} %where $\Psi (F)= \E[\abs{F}^4]-2 \big(\E[\abs{F}^2]\big)^2$. %Please refer to \cite[p. 210]{Nour Pecc12} for the definition of the  Wassersterin distance $d_{\mathrm{W}}$.
	\end{remark}
	
	\subsection{Multivariate complex FMT and Berry-Ess\'{e}en bound}\label{section 3.2}
	
	Corollary 4.7 and Proposition 4.9 of \cite{Camp15} imply a multivariate complex FMT, but the assumption of the following Theorem~\ref{thm Camp} is more simple and our proof is completely different to that of \cite{Camp15}. 
	\begin{theorem}[multivariate FMT] \label{thm Camp}
		For an integer $d\ge 1$, let $F_n=(F_n^1,F_n^2,\dots, F_n^d)'$ be a sequence of random
		vectors, and let $(p_1,q_1),\dots,(p_d,q_d) $ be positive integer pairs such that for each $j = 1, \dots, d$,
		$F_n^j = I_{p_j,q_j} ( f_n^j)$ for some kernel $f_n^j \in \FH^{\odot p_j}\otimes \FH^{\odot q_j} $. Moreover, assume that  $F_n$  is circularly-symmetric %invariant  by the action of $\mathbb{T}$ 
		and as $n\to \infty$,
		\begin{align} \label{Fj bar Fk}
			\E[F_n^j\overline{F_n^i}]&\to \E[Z_j\bar{Z}_i]:=\Sigma(j,i),\quad 1\le j,i \le d.
			%E[\abs{F_n^j {F_n^k}}^2]&\to E[\abs{Z_j {Z}_k}^2],
		\end{align}
		Then, as $n\to \infty$, the following two conditions are equivalent: 
		\begin{itemize}
			\item[\textup{(i)}]
			$(F_n)$ converges in law to $Z\sim \mathcal{CN}_d(0,\Sigma)$.
			\item[\textup{(ii)}] For every $1 \le j  \le d$, $(F^j_n)$ converges in law to $\mathcal{CN}(0,\Sigma(j,j))$.
		\end{itemize}
	\end{theorem}

	The following complex chaotic central limit theorem (see \cite{Hu17, HuNua05, NourPecc12} for the real version) is a corollary of Theorem~\ref{thm Camp}.
	\begin{theorem}\label{thm Pecca TUd}
		Let $\set {F_n:\,n\ge 1}$ be a sequence in $L^2(\Omega)$ with zero means such that 
		$F_n$ has the following Wiener-It\^{o} chaos decomposition 
		$$F_n=\sum_{p+q\ge 1}I_{p,q}( f_{n,p,q}),$$ where $\left\lbrace I_{p,q}( f_{n,p,q}): p+q\geq 1\right\rbrace $ satisfies that for any $M\geq 1$, the random vector composed by $\left\lbrace I_{p,q}( f_{n,p,q}): 1\leq p+q\leq M \right\rbrace $ is circularly-symmetric. Assume the following conditions hold:
		\begin{itemize}
			\item[\textup{(i)}] For each $p+q\ge 1$,  $\lim\limits_{n\to \infty} p!q!\norm{f_{n,p,q}}^2_{\FH^{\otimes (p+q)}}= \sigma^2_{p,q}\ge 0$.
			\item[\textup{(ii)}] $\sigma^2=\sum\limits_{p+q\ge 1} \sigma^2_{p,q}<\infty $,
			\item[\textup{(iii)}] For each $p+q\ge 2$,  and $1\le i+j\le p+q-1,\,0\le i\le p,\,0\le j\le q$, $\lim\limits_{n\to \infty} \norm{f_{n,p,q} \otimes_{i,j} {h_{n,q,p}}}^2_{\FH^{\otimes (2(p+q-i-j)}}=0$, where 
			$  {h_{n,q,p}} $ is the reverse complex conjugate of function $f_{n,p,q}$.
			\item[\textup{(iv)}] $\lim\limits_{M\to \infty} \sup\limits_{n\ge 1}\sum\limits_{p+q> M}  p!q!\norm{f_{n,p,q}}^2_{\FH^{\otimes (p+q)}}=0$.
		\end{itemize}
		Then  $ F_n \stackrel{ {law}}{\to} \mathcal{CN}(0,\sigma^2)$ as $n\to \infty$.
	\end{theorem}
	
	\begin{remark}
		The assumption of circular symmetry of $\left\lbrace I_{p,q}( f_{n,p,q}): 1\leq p+q\leq M \right\rbrace $ in Theorem \ref{thm Pecca TUd} is necessary to guarantee that we can use Theorem \ref{thm Camp} to get the conclusion. Note that for a $d$-dimensional complex-valued random vector $Z=(Z_1,Z_2,\dots, Z_d)'$, $Z_j$ is circularly symmetrical and $Z_j, Z_k$ are orthogonal (namely $\E\left[ Z_i\bar{Z_j}\right] =0$) for $1\leq j,k\leq d$ can not imply that $Z$ is circularly symmetrical, see \cite[Example 3.7.7]{GRG13} for an instance. This fact tells us that the circular symmetry and isometry property \eqref{isometry property} of $I_{p,q}( f_{n,p,q})$ for $p,q$ with $p+q\geq 1$ can not imply that the random vector composed by $\left\lbrace I_{p,q}( f_{n,p,q}): 1\leq p+q\leq M \right\rbrace $ is circularly-symmetric for any $M\geq 1$.
	\end{remark}
	\begin{remark}
		See \cite{Hu17} and \cite[Theorem 3]{ HuNua05} for the real version of Theorem \ref{thm Pecca TUd}. 
	\end{remark}
	%------------------------------------------------------
	
	A remarkable fact about the Berry-Ess\'{e}en upper bound for the multivariate complex fourth moment is related to a partially order relation of the indices of the components of a Wiener chaos random vector. Clearly, this fact has no real counterparts. 
	
	Suppose that $(p_j,q_j)\in \N^2,\,j=1,2$.  If $(p_1,\,q_1)\neq (p_2,\,q_2) $ and $p_1\ge p_2,\,q_1\ge q_2$, then we denote by $$(p_1,\,q_1)\curlyeqsucc (p_2,\,q_2),$$ or $$(p_2,\,q_2)\curlyeqprec (p_1,\,q_1).$$ If we remove the requirement $(p_1,\,q_1)\neq (p_2,\,q_2) $ in the above definition then `$\curlyeqsucc  $' gives a partially order relation in terms of set theory. It is obvious that 
	\begin{equation*}
		\mathbb{1}_{\set{ (p_2,q_2)\curlyeqsucc (p_1,q_1)}}=\mathbb{1}_{\set{p_1< p_2,q_1\le q_2}}+\mathbb{1}_{\set{p_1=p_2,q_1<q_2}} =\mathbb{1}_{\set{p_1\le p_2,q_1< q_2}}+\mathbb{1}_{\set{p_1<p_2,q_1=q_2}},
	\end{equation*}and 
	\begin{equation*}
		\mathbb{1}_{\set{ (p_1,q_1)\neq (p_2,q_2)}}=\mathbb{1}_{\set{(p_2,q_2)\curlyeqsucc (p_1,q_1)}} +\mathbb{1}_{\set{(p_1,q_1)\curlyeqsucc  (p_2,q_2)}} +\mathbb{1}_{\set{p_1>p_2,q_1<q_2}}+\mathbb{1}_{\set{p_1<p_2,q_1>q_2}}.
	\end{equation*}
	%------------------------------------------------------
	
	%------------------------------------------------------
	\begin{thm}\label{duowei b-e bound thm}
		For an integer $d\ge 1$, let $F=(F^1,\dots,F^d)'$ with each component $F^j = I_{p_j,q_j} ( f^j)$ where $f^j \in \FH^{\odot p_j}\otimes \FH^{\odot q_j} $ and $(p_j,q_j)$ is positive integer pair for $j = 1, \dots, d$. Assume that $F$ is circularly-symmetric and $\E[F\bar{F}']=\Sigma$. Assume that $Z\sim \mathcal{CN}(0,\Sigma)$. If, in addition, $\Sigma$ is invertible, then  the Wasserstein distance between the laws of $F$ and $Z$ satisfies the following fourth moment type Berry-Ess\'{e}en upper bound:
		\begin{align}\label{Fourth moment BEB2 coro 2}
			d_{W}(F, Z)\le { \frac{2\sqrt{d\lambda_{max}}}{\lambda_{min}}}\sqrt{\E\norm{F}^4-\E\norm{N}^4},
		\end{align}where $\lambda_{max},\,\lambda_{min}$ are respectively the maximum and minimum eigenvalue of $\Sigma$ and the norm $\norm{\cdot}$ is the Euclid norm of $\C^d$. Moreover, there exists a constant $c$ depending on $(p_k,q_k)$ such that 
	\begin{equation}\label{key estimate 000}
		\begin{aligned}
			\E\norm{F}^4-\E\norm{N}^4  
			&\le c \sum_{r=1}^{d}\Big\{ \sum_{0<i+i'<l_r}\norm{f_r {\otimes}_{i,i'}h_r}^2_{\FH^{\otimes(2(l_r-i-i'))}}  \\
			& \quad +   \sum_{j\neq r,\,1\le j\le d} \Big[\mathbb{1}_{\set{ (p_j,q_j)\curlyeqsucc (q_r,p_r)}} \norm{f_r}^2_{\FH^{\otimes l_r}}\cdot\norm{f_j\otimes_{p_j-q_r,q_j-p_r}h_j}_{\FH^{\otimes{2l_r} }} \\
			&\quad +\mathbb{1}_{\set{(p_j,q_j)\curlyeqsucc (p_r,q_r)}}\norm{f_r}^2_{\FH^{\otimes l_r}}\cdot\norm{f_j\otimes_{p_j-p_r,q_j-q_r}h_j}_{\FH^{\otimes{2l_r} }} \\
			&\quad +\mathbb{1}_{\set{(p_r,\,q_r)\curlyeqsucc (q_j,\,p_j)}}\norm{f_j }^2_{\FH^{\otimes l_j}}\cdot\norm{f_r\otimes_{p_r-q_j,\,q_r-p_j}h_r}_{\FH^{\otimes{2l_j} }}  \\
			&\quad +\mathbb{1}_{\set{(p_r,\,q_r)\curlyeqsucc (p_j,\,q_j)}}\norm{f_j }^2_{\FH^{\otimes l_j}}\cdot\norm{f_r\otimes_{p_r-p_j,\,q_r-q_j}h_r}_{\FH^{\otimes{2l_j} }}  \Big]\Big\}.
		\end{aligned}
	\end{equation}
		where $l_r=p_r+q_r$, $f_r$ and $h_r$ are respectively the kernels of $F^r$ and $\bar{F^r}$ for $1\leq r\leq d$.
	\end{thm}
	\begin{remark}\label{remark3.12}
		\begin{itemize}
			\item[\textup{(i)}] The extension to the general case of $\E[F\bar{F}']  \neq \Sigma$ can be deduced by a straightforward adaptation of our arguments.
			\item[\textup{(ii)}]  The estimate \eqref{key estimate 000} has a very special feature. For example, if $d=2$ and $(p_1,q_1)=(5,1),\,(p_2,q_2)=(3,2)$, then in the estimate \eqref{key estimate 000}, the second term $\sum_{j\neq k}$ vanishes. On the other hand, if we use the identity \eqref{pp2 Connection} together with Theorem 6.2.2  of \cite{NourPecc12}, the estimate of the real Wiener chaos,  some extra terms similar to that in $\sum_{j\neq k}$ appear.  That is to say, the estimate \eqref{key estimate 000}  may improve the corresponding estimate of \cite[Theorem 6.2.2]{NourPecc12}.
		\end{itemize}
	\end{remark}
	
	\subsection{Application}\label{section 3.3}
	In this section, we apply our results to get the optimal Berry-Ess\'een bound of $F_T$ defined as \eqref{FT dyi}. Without loss of generality, we assume $a=1$ in \eqref{cp}. For readers' convenience, we restate the complex Ornstein-Uhlenbeck process which models the Chandler wobble
	\begin{equation}\label{SDE}
		\mathrm{d}Z_t=-\gamma Z_t\mathrm{d}t+ \mathrm{d}\zeta_t, \quad t\geq0,
	\end{equation}
	where $Z_t$ is a complex-valued process, $Z_0=0$, $\gamma \in\mathbb{C}$ is unknown, $\lambda:=\mathrm{Re}\;\gamma > 0$, and $\zeta_t$ is a complex fractional Brownian motion with Hurst parameter $H\in [\frac12,\frac34)$. That is, $\zeta_t =\frac{ B_t^1+\mi B_t^2}{\sqrt2}$, where
	$(B_t^1, B_t^2)$ is a two-dimensional fractional Brownian motion with Hurst parameter $H\in [\frac12,\frac34)$. Suppose that only one trajectory $\left( Z_t, 0\leq t\leq T\right) $ can be observed. By minimizing $\int_{0}^{T}\left|\dot{Z}_{t}+\gamma Z_{t}\right|^{2} \mathrm{d} t$, one can obtain a least squares estimator of $\gamma$ as follows,
	\begin{equation*}
		\hat{\gamma}_{T}=-\frac{\int_{0}^{T} \bar{Z}_{t} \mathrm{d} Z_{t}}{\int_{0}^{T}\left|Z_{t}\right|^{2} \mathrm{d} t}=\gamma- \frac{\int_{0}^{T} \bar{Z}_{t} \mathrm{d} \zeta_{t}}{\int_{0}^{T}\left|Z_{t}\right|^{2} \mathrm{d} t}.
	\end{equation*}
	Here, in order to derive the optimal Berry-Ess\'een bound of $F_T$ (see Theorem \ref{example_BM}), we consider the special case $H=\frac12$, namely, the case that noise is a complex standard Brownian motion. 
	
	In \cite{chw17}, Chen, Hu and Wang proved that $\sqrt{T}\left(\hat{\gamma}_{T}-\gamma\right)$ is asymptotically normal. Namely, as $ T \rightarrow \infty$,
	\begin{equation*}
		\sqrt{T}\left[\hat{\gamma}_{T}-\gamma\right]=-\frac{\frac{1}{\sqrt{T}}\int_{0}^{T} \bar{Z}_{t} \mathrm{d} \zeta_{t}}{\frac{1}{T}\int_{0}^{T}\left|Z_{t}\right|^{2} \mathrm{d} t} \overset{d}{\rightarrow} \mathcal{CN}\left(0, 2\lambda\right),
	\end{equation*}
	where $\mathrm{Id}_2$ denotes $2\times2$ identity matrix.
	They showed that denominator satisfies
	\begin{equation*}
		\frac{1}{T}\int_{0}^{T}\left|Z_{t}\right|^{2} \mathrm{d} t\overset{a.s.}{\rightarrow}\frac{1}{2\lambda},
	\end{equation*} 
	and for numerator $F_T:=\frac{1}{\sqrt{T}}\int_{0}^{T} \bar{Z}_{t} \mathrm{d} \zeta_{t}$,
	\begin{equation}\label{example}
		F_T\overset{d}{\rightarrow}\mathcal{CN}\left(0, \frac{1}{2\lambda}\right).
	\end{equation} 
	Then the asymptotic normality of the estimator $\hat{\gamma}_{T}$ is obtained.

	Define the Hilbert space $\mathfrak{H}=L_{\mathbb{C}}^2\left( \left[ 0,+\infty\right) \right) $ with inner product $\left\langle  f,g\right\rangle_{\mathfrak{H}}= \int_{0}^{\infty}  f(t) \overline{g(t)} \mathrm{d} t$ for any $f,g\in\mathfrak{H}$. Given $f\in\mathfrak{H}^{\odot a}\otimes\mathfrak{H}^{\odot b}$, $g\in\mathfrak{H}^{\odot c}\otimes\mathfrak{H}^{\odot d}$, for $i=0,\dots,a\land d$, $j=0,\dots,b\land c$, the $(i,j)$-th contraction of $f$ and $g$ is the element of $\mathfrak{H}^{\odot (a+c-i-j)}\otimes\mathfrak{H}^{\odot (b+d-i-j)}$ defined by
	\begin{align*}
		&f \otimes_{i, j} g\left(t_{1}, \ldots, t_{a+c-i-j} ; s_{1}, \ldots, s_{b+d-i-j}\right) \\
		=&\,\int_{\mathbb{R}_{+}^{2 l}} f\left(t_{1}, \ldots, t_{a-i}, u_{1}, \ldots, u_{i} ; s_{1} \ldots, s_{b-j}, v_{1} ,\ldots, v_{j}\right) \\
		&\qquad \quad g\left(t_{a-i+1}, \ldots, t_{p-l}, v_{1}, \ldots, v_{j}; s_{b-j+1}, \ldots, s_{q-l}, u_{1}, \ldots, u_{i}\right) \mathrm{d} \vec{u} \mathrm{d} \vec{v} ,
	\end{align*}
	where $l=i+j, p=a+c, q=b+d, \vec{u}=\left(u_{1}, \ldots, u_{i}\right)$ and $\vec{v}=\left(v_{1}, \ldots, v_{j}\right)$.
	
	According to \eqref{SDE}, we know that
	\begin{equation}\label{F_T}
		F_T=\frac{1}{\sqrt{T}}\int_{0}^{T} \bar{Z}_{t} \mathrm{d} \zeta_{t}=\frac{1}{\sqrt{T}}\int_{0}^{T} \int_{0}^{T}e^{-\bar{\gamma}(t-s)}\mathbf{1}_{\left\lbrace 0\leq s\leq t\leq T\right\rbrace} \mathrm{d} \zeta_{t}\mathrm{d}\bar{\zeta}_s=I_{1,1}(\psi_{T}(t,s)),
	\end{equation}
	where
	\begin{equation*}
		\psi_{T}(t,s)=\frac{1}{\sqrt{T}} e^{-\bar{\gamma}(t-s)}\mathbf{1}_{\left\lbrace 0\leq s\leq t\leq T\right\rbrace}.
	\end{equation*}
	Then
	\begin{equation*}
		\bar{F_T}=I_{1,1}(h_{T}(t,s)),
	\end{equation*}
	where
	\begin{equation*}
		h_{T}(t,s)=\overline{\psi_{T}(s,t)}=\frac{1}{\sqrt{T}} e^{-\gamma(s-t)}\mathbf{1}_{\left\lbrace 0\leq t\leq s\leq T\right\rbrace}.
	\end{equation*}
	By isometry property of complex Wiener It\^o integral \eqref{isometry property}, we obtain that
	\begin{equation*}
		\begin{aligned}
			\E\left[ F_T^2\right] &=\E(I_{1,1}(\frac{1}{\sqrt{T}}\psi_{T})\overline{I_{1,1}(\frac{1}{\sqrt{T}}h_T)})=\frac{1}{T}\left\langle \psi_{T},h_T \right\rangle _{\mathfrak{H}^{\otimes 2}}\\
			&=\frac{1}{T}\int_{0}^{\infty}\int_{0}^{\infty} \psi_{T}\left(t,s \right)  \overline{h_{T}\left( t,s\right) }\mathrm{d}t\mathrm{d}s\\
			&=\frac{1}{T}\int_{0}^{\infty}\int_{0}^{\infty} e^{-\bar{\gamma}(t-s)}\mathbf{1}_{\left\lbrace 0\leq s\leq t\leq T\right\rbrace} e^{-\bar{\gamma}(s-t)}\mathbf{1}_{\left\lbrace 0\leq t\leq s\leq T\right\rbrace}\mathrm{d}t\mathrm{d}s=0,
		\end{aligned}
	\end{equation*}
	and
	\begin{equation*}
		\begin{aligned}
			\E\left[ |F_T|^2\right] & =\E(I_{1,1}(\frac{1}{\sqrt{T}}\psi_{T})\overline{I_{1,1}(\frac{1}{\sqrt{T}}\psi_{T})})=\frac{1}{T}\left\langle \psi_{T},\psi_{T} \right\rangle _{\mathfrak{H}^{\otimes 2}}\\
			&=\frac{1}{T}\int_{0}^{\infty}\int_{0}^{\infty} \psi_{T}\left(t,s \right)  \overline{\psi_{T}\left( t,s\right) }\mathrm{d}t\mathrm{d}s
			\\&=\frac{1}{T}\int_{0}^{\infty}\int_{0}^{\infty} e^{-\bar{\gamma}(t-s)}\mathbf{1}_{\left\lbrace 0\leq s\leq t\leq T\right\rbrace} e^{-\gamma(t-s)}\mathbf{1}_{\left\lbrace 0\leq s\leq t\leq T\right\rbrace}\mathrm{d}t\mathrm{d}s
			\\&=\frac{1}{T}\int_{0}^{T}\int_{0}^{t}e^{-2\lambda(t-s)}\mathrm{d}s\mathrm{d}t=\frac{1}{2\lambda}+\frac{1}{4\lambda^2T}e^{-2\lambda T}-\frac{1}{4\lambda^2T}
			\\&\rightarrow\frac{1}{2\lambda}\text{ as }T\rightarrow\infty.  
		\end{aligned}
	\end{equation*}
	Note that $\lim\limits_{T\rightarrow \infty}\left( 1+\frac{1}{2\lambda T}e^{-2\lambda T}-\frac{1}{2\lambda T}\right)=1 $, then $1+\frac{1}{2\lambda T}e^{-2\lambda T}-\frac{1}{2\lambda T}>0$ for $T$ large enough. Now we consider 
	\begin{equation}\label{normalized F_T}
		\tilde{F_T}=\left(1+\frac{1}{2\lambda T}e^{-2\lambda T}-\frac{1}{2\lambda T} \right) ^{-\frac{1}{2}}F_T.
	\end{equation}
	and derive the optimal Berry-Ess\'een bound of $\tilde{F_T}$ converging in law to a complex normal random variable $Z\sim\mathcal{CN}(0,\frac{1}{2\lambda})$ as $T\rightarrow\infty$. According to \cite[Lemma 4.6, Lemma 4.7]{c23} and Proposition \ref{be bound prop 63}, we have that 
	\begin{align}\label{example_3moment}
		\left| \left\langle f\tilde{\otimes}_{0,1}f,h \right\rangle _{\FH^{\otimes 2}}+\left\langle f\tilde{\otimes}_{1,0}f,h \right\rangle _{\FH^{\otimes 2}}\right|&=\left| \E\left[F_T^3 \right] \right|=0, \\
		\left|  \left\langle f\tilde{\otimes}_{0,1}f,f \right\rangle _{\FH^{\otimes 2}}+\left\langle f\tilde{\otimes}_{1,0}f,f \right\rangle _{\FH^{\otimes 2}}\right| &= \left|\E\left[F_T^2\bar{F_T} \right] \right|  \asymp \frac{1}{\sqrt{T}},
	\end{align}
	and 
	\begin{equation}\label{example_4moment}
		\E\left[\left| F_T \right| ^4 \right]-2\left(\E\left[ \left| F_T \right| ^2\right]  \right) ^2-\left| \E\left[ F_T^2\right] \right| ^2\asymp \norm{f {\otimes}_{0,1}h}^2_{\FH^{\otimes2}}+ \norm{f {\otimes}_{1,0}h}^2_{\FH^{\otimes2}}\asymp\frac{1}{T}.
	\end{equation}

	Combining Theorem \ref{cFMBE bound}, \eqref{example_3moment} and \eqref{example_4moment}, we obtain the following theorem.
	\begin{theorem}\label{example_BM}
		Suppose that $\tilde{F_T}$ is defined as \eqref{normalized F_T} and $Z\sim\mathcal{CN}(0,\frac{1}{2\lambda})$. Then 
		\begin{equation*}
			c_1(\lambda) 
			\frac{1}{\sqrt{T}}\leq d_{W}(\tilde{F_T}, Z)\leq c_2(\lambda)\frac{1}{\sqrt{T}},
		\end{equation*}
		where $c_1(\lambda) $ and $c_2(\lambda) $ are two constants only depending on $\lambda$.
	\end{theorem}
	\begin{rem}
		For $H\in (\frac12,\frac34)$, from \cite{cy23,chw17}, we know that
		\begin{equation*}
			(\RE F_T, \IM F_T)\overset{d}{\rightarrow}\mathcal{N}_2\left(0, \alpha_H^2 \tensor{C} \right),
		\end{equation*}
		where $\alpha_H=H(2H-1)$ and  $\tensor{C}=\frac{1}{2}\left[\begin{array}{cc}
			\sigma^{2}+a & b \\
			b & \sigma^{2}-a
		\end{array}\right]$ with
		\begin{equation*}
			\begin{aligned}
				\sigma^{2}&=\frac{2}{\Gamma(2-2 H)^{2}} \int_{[0, \infty)^{2}} \mathrm{d} x \mathrm{d} y \frac{(x y)^{1-2 H}}{(x+y)(x+\bar{\gamma})(y+\gamma)}\\&\quad
				+\frac{\Gamma^{2}(2 H-1)}{2 \lambda}\left(\frac{2}{|\gamma|^{4 H-2}}+\frac{1}{\gamma^{4 H-2}}+\frac{1}{\bar{\gamma}^{4 H-2}}\right)\\&= \frac{\Gamma(2H-1)}{2
					\lambda}\left( \frac{1}{\gamma^{4H-2}}+\frac{1}{\bar{\gamma}^{4H-2}}\right)  \left(  \Gamma(2H-1) +\frac{2\Gamma(3-4H)\Gamma(4H-2)}{\Gamma(2-2H)}\right) ,
				\\a+\mathrm{i} b&=\frac{2}{\Gamma(2-2 H)^{2}} \int_{[0, \infty)^{2}} \frac{(x y)^{1-2 H}}{(y+\gamma)^{2}}\left[\frac{1}{x+y}+\frac{1}{x+\gamma}\right] \mathrm{d} x \mathrm{d} y\\&= \frac{2\Gamma(2H-1)}{ 
					{\bar{\gamma}}^{4H-1}} \left(  \Gamma(2H) +\frac{\Gamma(3-4H)\Gamma(4H-1)}{\Gamma(2-2H)}\right)  .
			\end{aligned}
		\end{equation*}
		Using \cite[Lemma 3.1]{ckl20}, we can derive that 
		\begin{equation*}
			\sqrt{E\left[\left| F_T \right| ^4 \right]-2\left(E\left[ \left| F_T \right| ^2\right]  \right) ^2-\left| E\left[ F_T^2\right] \right| ^2 }\leq c(\lambda,H) \begin{cases}T^{-\frac{1}{2}}, & H \in\left[\frac{1}{2}, \frac{5}{8}\right), \\ T^{-\left( \frac{1}{2}-\right) }, & H=\frac{5}{8}, \\ T^{4 H-3}, & H \in\left(\frac{5}{8}, \frac{3}{4}\right),\end{cases}
		\end{equation*}
		where $c(\lambda,H)$ is a constant depending on $\lambda$ and $H$. Denote by $\tensor{C}_T$ the covariance matrix of $	(\RE F_T, \IM F_T)$. Note that the matrix $\tensor{C}$ is invertible, then for $T$ large enough, $\tensor{C}_T$ is invertible. Let $\tilde{\tilde{F_T}}=\tensor{C}^{\frac12}\tensor{C}_T^{-\frac12} F_T$ and $N$ be a bivariate normal random vector with mean zero and covariance matrix $\tensor{C}$, then we can get the Berry-Ess\'een upper bound of $\tilde{\tilde{F_T}}$ as 
		\begin{equation*}
			d_W\left( \tilde{\tilde{F_T}},N\right) \leq c(\lambda, H, a,b,\sigma)\begin{cases}T^{-\frac{1}{2}}, & H \in\left[\frac{1}{2}, \frac{5}{8}\right), \\ T^{-\left( \frac{1}{2}-\right) }, & H=\frac{5}{8}, \\ T^{4 H-3}, & H \in\left(\frac{5}{8}, \frac{3}{4}\right),\end{cases}
		\end{equation*}
		where $c(\lambda, H, a,b,\sigma)$ is a constant depending on $\lambda, H, a, b$ and $\sigma$. In order to obtain the Berry-Ess\'een lower bound provided as \eqref{lower bound} in Theorem \ref{cFMBE bound}, more precise calculations of the quantities $\left| \E\left[F_T^3 \right] \right|$, $\left|\E\left[F_T^2\bar{F_T} \right] \right|$ and $E\left[\left| F_T \right| ^4 \right]-2\left(E\left[ \left| F_T \right| ^2\right]  \right) ^2-\left| E\left[ F_T^2\right] \right| ^2$ are required. This improving topic will be investigated in other works.
	\end{rem}
	
	\section{Proofs of the main results}\label{section 4}
	
	\subsection{Some technical estimates}\label{section 4.1}
	
	\begin{lemma}\label{basic inequliayt}
		Suppose that $f_j\in \FH^{\odot p_j}\otimes \FH^{\odot q_j}$ with $j=1,2$. Denote $l_j=p_j+q_j$ and $h_j$ the reverse complex conjugate of $f_j$. Then for $0\leq i\leq p_1\wedge q_2$ and $0\leq j\leq q_1\wedge p_2$,
		\begin{align}
			\norm{f_1 {\otimes}_{i,j} f_2}_{\FH^{\otimes(l_1+l_2-2(i+j))}}&=\norm{f_2 {\otimes}_{j,i} f_1}_{\FH^{\otimes(l_1+l_2-2(i+j))}},\label{communicative law}\\
			\norm{f_1\tilde{\otimes}_{i,j} f_2}_{\FH^{\otimes(l_1+l_2-2(i+j))}}&\le \norm{f_1\otimes_{i,j} f_2}_{\FH^{\otimes(l_1+l_2-2(i+j))}}\le \norm{f_1}_{\FH^{\otimes l_1}}\cdot\norm{f_2}_{\FH^{\otimes l_2}},\label{contraction symmetry inequality}\\
			\norm{f_1\otimes_{i,j} f_2}^2_{\FH^{\otimes(l_1+l_2-2(i+j))}}&=\innp{f_1\otimes_{p_1-i,q_1-j} h_1,\,h_2\otimes_{q_2-i, p_2-j} f_2 }_{\FH^{\otimes(2(i+j))}}.\label{fubi identity} \end{align}
	\end{lemma}
	\begin{proof} 
		The identity \eqref{communicative law} is trivial. The first inequality in \eqref{contraction symmetry inequality} is well-known and from Minkowski's inequality. The second inequality in \eqref{contraction symmetry inequality} is an immediate consequence of Cauchy-Schwarz inequality. The identity \eqref{fubi identity} can be obtained from the Fubini theorem.
	\end{proof}
	\begin{lemma}\label{norm inequality lemma}
		Let $f_j,\,j=1,2$ be as in Lemma~\ref{basic inequliayt}. Then for $0\leq i\leq p_1\wedge q_2$ and $0\leq j\leq q_1\wedge p_2$,
		\begin{equation}\label{norm inequality}
		\begin{aligned}
			2 \norm{f_1 {\otimes}_{i,j} f_2 }^2_{\FH^{\otimes ( l_1+l_2-2(i+j))}} &\le 2 \norm{f_1 {\otimes}_{p_1-i,q_1-j} h_1 }_{\FH^{\otimes ( 2( i+j))}} \cdot\norm{f_2 {\otimes}_{p_2-j,q_2-i} h_2 }_{\FH^{\otimes ( 2( i+j))}}\\
			&\le  \norm{f_1 {\otimes}_{p_1-i,q_1-j} h_1 }^2_{\FH^{\otimes ( 2( i+j))}} +\norm{f_2 {\otimes}_{p_2-j,q_2-i} h_2 }^2_{\FH^{\otimes ( 2( i+j))}}.
		\end{aligned}
	\end{equation}
		Especially, suppose that $f \in \FH^{\odot p}\otimes \FH^{\odot q}$, $h$ is the reverse complex conjugate of $f$ and $l=p+q$. Then for any $0\leq i,j\le p\wedge q$, %\, 0<i+j\le 2(p\wedge q)-1$,
		\begin{equation}\label{norm inequality 0}
			2 \norm{f {\otimes}_{i,j} f }^2_{\FH^{\otimes ( 2(l-i-j))}} \le  \norm{f {\otimes}_{p-i,q-j} h }^2_{\FH^{\otimes ( 2( i+j))}} +\norm{f {\otimes}_{p-j,q-i} h }^2_{\FH^{\otimes ( 2( i+j))}}.
		\end{equation}
	\end{lemma}
	\begin{proof}
		To show \eqref{norm inequality}, we apply first the identity  \eqref{fubi identity}, and then Cauchy-Schwarz inequality, the identity \eqref{communicative law} and the inequality $2ab\le a^2+b^2$. \eqref{norm inequality 0} is an immediate consequence of \eqref{norm inequality}.
	\end{proof}

	\begin{prop}\label{jproo key estimate}
		Denote $F_j=I_{p_j,q_j}(f_j),\,j=1,2$ and $l_j=p_j+q_j$. %Suppose that $F_j$  is circularly symmetric, % and  $l_1\ge l_2$, 
		Then there exists a constant $c$ depending on $p_j,q_j$ such that 
		\begin{equation}\label{key estimate}
		\begin{aligned}
			\quad & \mathrm{Cov}(\abs{F_1}^2,\, \abs{F_2}^2)-\abs{\E[F_1\bar{F}_2]}^2-\abs{\E[F_1F_2]}^2\\
			\le&\, c\Big[ \sum_{0<i+j<l_1}\norm{f_1 {\otimes}_{i,j}h_1}^2_{\FH^{\otimes(2(l_1-i-j))}}+\sum_{0<i+j<l_2}\norm{f_2 {\otimes}_{i,j}h_2}^2_{\FH^{\otimes(2(l_2-i-j))}} \\
			& +  \mathbb{1}_{\set{ (p_2,q_2)\curlyeqsucc (q_1,p_1)}} \norm{f_1}^2_{\FH^{\otimes l_1}}\cdot\norm{f_2\otimes_{p_2-q_1,q_2-p_1}h_2}_{\FH^{\otimes{2l_1} }} \\
			& +\mathbb{1}_{\set{ (p_2,q_2)\curlyeqsucc (p_1,q_1)}}\norm{f_1}^2_{\FH^{\otimes l_1}}\cdot\norm{f_2\otimes_{p_2-p_1,q_2-q_1}h_2}_{\FH^{\otimes{2l_1} }} \\
			& +\mathbb{1}_{\set{(p_1,\,q_1)\curlyeqsucc (q_2,\,p_2)}}\norm{f_2 }^2_{\FH^{\otimes l_2}}\cdot\norm{f_1\otimes_{p_1-q_2,q_1-p_2}h_1}_{\FH^{\otimes{2l_2} }}  \\
			& +\mathbb{1}_{\set{(p_1,\,q_1)\curlyeqsucc (p_2,\,q_2)}}\norm{f_2 }^2_{\FH^{\otimes l_2}}\cdot\norm{f_1\otimes_{p_1-p_2,q_1-q_2}h_1}_{\FH^{\otimes{2l_2} }}  \Big].
		\end{aligned}
	\end{equation}
	\end{prop}

	%\begin{remark}
	%\blue{Here we should compare it with the real case carefully .}
	%\end{remark}
	\begin{proof}
		%The assumption of $F_j$ circular symmety implies that 
		%\begin{align*}
		%\E[F_1^2]=\E[F_2^2]=\E[F_1F_2]=0.
		%\end{align*} 
		%Hence, 
		Lemma 3.1 of \cite{ch17} implies that 
		\begin{equation}\label{express 4moment}
		\begin{aligned}
			&\quad \mathrm{Cov}(\abs{F_1}^2,\, \abs{F_2}^2)-\abs{\E[F_1\bar{F}_2]}^2-\abs{\E[F_1F_2]}^2 \\
			&=\sum_{0 <k+k'< l}{p_1\choose k}{q_1\choose k'}{q_2\choose k'}{p_2\choose k}  p_1!q_1! p_2!q_2!  \norm{f_1\otimes_{k,k'}h_2}^2_{\FH^{\otimes(m-2(k+k'))}} \\
			&\quad + \mathbb{1}_{(p_1,q_1)\neq (p_2,q_2)}\mathbb{1}_{k+k'= l} {p_1\choose k}{q_1\choose k'}{q_2\choose k'}{p_2\choose k}  p_1!q_1! p_2!q_2!  \norm{f_1\otimes_{k,k'}h_2}^2_{\FH^{\otimes(m-2l)}} \\
			&\quad +\sum_{r=1}^{l'-1}(p_1+p_2-r)!(q_1+q_2-r)!\norm{\phi_r}^2_{\FH^{\otimes(m-2r)}}\\
			&\quad + \mathbb{1}_{(p_1,q_1)\neq (q_2,p_2)}(p_1+p_2-l')!(q_1+q_2-l')! \norm{\phi_{l'}}^2_{\FH^{\otimes(m-2l')}} , 
		\end{aligned} 
	\end{equation}
	where $m=l_1+l_2$, $l'=p_1\wedge q_2+ q_1\wedge p_2$, $l= p_1\wedge p_2+ q_1\wedge q_2$ and 
		\begin{align*}
			\phi_r=\sum_{i+j=r} {p_1\choose i}{q_1\choose j}{q_2\choose i }{p_2\choose j}  i!j!\, f_1\tilde{\otimes}_{i,j} f_2, \mbox{ for } 1\leq r\leq l'.
		\end{align*}
		For the first term of right hand side of \eqref{express 4moment}, applying \eqref{communicative law} and \eqref{norm inequality}, we have that
		\begin{equation}\label{first term inequality}
		\begin{aligned}
			&\sum_{0 <k+k'< l}{p_1\choose k}{q_1\choose k'}{q_2\choose k'}{p_2\choose k}  p_1!q_1! p_2!q_2! \norm{f_1\otimes_{k,k'}h_2}^2_{\FH^{\otimes(m-2(k+k'))}} \\
			\le &\,c\sum_{0 <k+k'< l} \norm{f_1\otimes_{p_1-k,q_1-k'}h_1}^2_{\FH^{\otimes(2(k+k'))}}+\norm{h_2\otimes_{q_2-k',p_2-k}f_2}^2_{\FH^{\otimes(2(k+k'))}} \\
			\le &\,c\sum_{0 <k+k'< l} \norm{f_1\otimes_{p_1-k,q_1-k'}h_1}^2_{\FH^{\otimes(2(k+k'))}}+\norm{f_2\otimes_{p_2-k,q_2-k'}h_2}^2_{\FH^{\otimes(2(k+k'))}}  \\
			\le &\,c\Big[ \sum_{0<i+j<l_1}\norm{f_1 {\otimes}_{i,j}h_1}^2_{\FH^{\otimes(2(l_1-i-j))}}+\sum_{0<i+j<l_2}\norm{f_2 {\otimes}_{i,j}h_2}^2_{\FH^{\otimes(2(l_2-i-j))}}\Big].
		\end{aligned}
	\end{equation}
		For the third term of right hand side of \eqref{express 4moment}, By Minkowski's inequality, power means inequality and similar argument of the proof of \eqref{first term inequality}, we have that
		\begin{equation}\label{third term}
		\begin{aligned}
			&\sum_{r=1}^{l'-1}(p_1+q_2-r)!(q_1+p_2-r)!\norm{\phi_r}^2_{\FH^{\otimes(m-2r)}}  \\
			\le&\, c \sum_{0 <k+k'< l'} \norm{f_1\otimes_{k,k'}f_2}^2_{\FH^{\otimes(m-2(k+k'))}}\\
			\le&\, c\Big[ \sum_{0<i+j<l_1}\norm{f_1 {\otimes}_{i,j}h_1}^2_{\FH^{\otimes(2(l_1-i-j))}}+\sum_{0<i+j<l_2}\norm{f_2 {\otimes}_{i,j}h_2}^2_{\FH^{\otimes(2(l_2-i-j))}}\Big].
		\end{aligned}
	\end{equation}
		For the second term of right hand side of \eqref{express 4moment}, applying Lemma \ref{norm inequality lemma} and Equation \eqref{communicative law}, we have that
		\begin{equation}\label{second term}
		\begin{aligned}
			& \mathbb{1}_{(p_1,q_1)\neq (p_2,q_2)}\mathbb{1}_{k+k'= l} {p_1\choose k}{q_1\choose k'}{q_2\choose k'}{p_2\choose k}  p_1!q_1! p_2!q_2!  \norm{f_1\otimes_{k,k'}h_2}^2_{\FH^{\otimes(m-2l)}}\\
			\le&\, c\Big[ \mathbb{1}_{\set{(p_2,\,q_2)\curlyeqsucc (p_1,\,q_1)}}\norm{f_1}^2_{\FH^{\otimes l_1}}\cdot\norm{f_2\otimes_{p_2-p_1,q_2-q_1}h_2}_{\FH^{\otimes{2l_1} }} \\
			& +\mathbb{1}_{\set{p_1>p_2,q_1<q_2}}(\norm{f_1\otimes_{p_1-p_2,0}h_1}^2_{\FH^{\otimes 2l}}+\norm{f_2\otimes_{0,q_2-q_1}h_2}^2_{\FH^{\otimes{2l} }})\\
			&+ \mathbb{1}_{\set{p_1<p_2,q_1>q_2}}(\norm{f_1\otimes_{0,q_1-q_2}h_1}^2_{\FH^{\otimes 2l}}+\norm{f_2\otimes_{p_2-p_1,0}h_2}^2_{\FH^{\otimes{2l} }})\\
			%&+\mathbb{1}_{\set{p_1<q_2,q_1=p_2}} \norm{f_1 }^2_{\FH^{\otimes l_1}}\cdot\norm{f_2\otimes_{0,q_2-p_1}h_2}^2_{\FH^{\otimes{2l_1} }}\\
			&+\mathbb{1}_{\set{(p_1,\,q_1)\curlyeqsucc (p_2,\,q_2)}}\norm{f_2 }^2_{\FH^{\otimes l_2}}\cdot\norm{f_1\otimes_{p_1-p_2,q_1-q_2}h_1}_{\FH^{\otimes{2l_2} }} \Big].
		\end{aligned}
	\end{equation}
		Similarly, for the fourth term of right hand side of \eqref{express 4moment}, we have that
		\begin{equation}\label{fourth term inequality}
		\begin{aligned}
			&\mathbb{1}_{(p_1,q_1)\neq (q_2,p_2)}(p_1+p_2-l')!(q_1+q_2-l')! \norm{\phi_{l'}}^2_{\FH^{\otimes(m-2l')}} \\
			\le&\, c\mathbb{1}_{(p_1,q_1)\neq (q_2,p_2)}\norm{f_1\otimes_{p_1\wedge q_2,p_2\wedge q_1}f_2}_{\FH^{\otimes{2l_1} }}^2 \\
			\le&\, c\Big[ \mathbb{1}_{\set{(q_2,\,p_2)\curlyeqsucc (p_1,\,q_1)}}\norm{f_1}^2_{\FH^{\otimes l_1}}\cdot\norm{f_2\otimes_{p_2-q_1,q_2-p_1}h_2}_{\FH^{\otimes{2l_1} }} \\
			&\quad +\mathbb{1}_{\set{p_1>q_2,q_1<p_2}}(\norm{f_1\otimes_{p_1-q_2,0}h_1}^2_{\FH^{\otimes 2l'}}+\norm{f_2\otimes_{ p_2-q_1,0}h_2}^2_{\FH^{\otimes{2l'} }})\\
			&\quad+ \mathbb{1}_{\set{p_1<q_2,q_1>p_2}}(\norm{f_1\otimes_{0,q_1-p_2}h_1}^2_{\FH^{\otimes 2l'}}+\norm{f_2\otimes_{0,q_2-p_1}h_2}^2_{\FH^{\otimes{2l'} }})\\
			&\quad +\mathbb{1}_{\set{(p_1,\,q_1)\curlyeqsucc (q_2,\,p_2)}}\norm{f_2 }^2_{\FH^{\otimes l_2}}\cdot\norm{f_1\otimes_{p_1-q_2,q_1-p_2}h_1}_{\FH^{\otimes{2l_2} }} \Big].
		\end{aligned}
	\end{equation}
		Substituting the inequalities \eqref{first term inequality}-\eqref{fourth term inequality} into the identity \eqref{express 4moment}, we obtain the desired estimate \eqref{key estimate}.
	\end{proof}

	\begin{proposition}\label{be bound prop 63}
		If $F=I_{p,q}(f)$ with $f\in  \FH^{\odot p}\otimes \FH^{\odot q}$ , %such that $\E[F^2]=0$, 
		then we have that
		\begin{equation}\label{contraction moment}
		\begin{aligned}
			c_1(p,q)   \sum_{0<i+j<l}\norm{f {\otimes}_{i,j}h}^2_{\FH^{\otimes(2(l-i-j))}}&\le \E[\abs{F}^4]-2 \big(\E[\abs{F}^2]\big)^2  -\abs{E[{F}^2]}^2\\&\le c_2(p,q)   \sum_{0<i+j<l}\norm{f {\otimes}_{i,j}h}^2_{\FH^{\otimes(2(l-i-j))}} ,
		\end{aligned}
	\end{equation}
		where $l=p+q$, $c_1(p,q)$ and $c_2(p,q)$ are two constants depending on $p,q$, and $h$ is the reverse complex conjugate of $f$.
	\end{proposition}
	%\begin{remark}\blue{It is a key estimate and we need a short proof.} \end{remark}
	\begin{proof}
		Taking $F_2=F_1=F$ in Proposition \ref{jproo key estimate} and Equation \eqref{express 4moment}, we obtain that 
		\begin{align*}
			\E[\abs{F}^4]-2 \big(\E[\abs{F}^2]\big)^2 -\abs{E[{F}^2]}^2
			&=\mathrm{Cov}(\abs{F}^2,\, \abs{F}^2)-\big(\E[\abs{F}^2]\big)^2-\abs{E[{F}^2]}^2\\
			&\leq  c_2(p,q)   \sum_{0<i+j<l}\norm{f {\otimes}_{i,j}h}^2_{\FH^{\otimes(2(l-i-j))}},
		\end{align*} 
		and 
		\begin{equation}\label{revised version1}
		\begin{aligned}
			&\quad\E[\abs{F}^4]-2 \big(\E[\abs{F}^2]\big)^2 -\abs{\E[{F}^2]}^2\\
			&=\sum_{0 <i+j<l}{p\choose i}^2{q\choose j}^2 (p!q!)^2  \norm{f\otimes_{i,j}h}^2_{\FH^{\otimes(2(l-i-j))}} \\
			&\quad +\sum_{r=1}^{l'-1}(2p-r)!(2q-r)!\norm{\varphi_r}^2_{\FH^{\otimes(2(l-r))}}\\
			&\quad+	\mathbf{1}_{\left\lbrace p\neq q\right\rbrace }(2p-l')!(2q-l')!\binom{p}{p\wedge q}^4\binom{q}{p\wedge q}^4(p\wedge q )!^4\norm{f\otimes_{p\wedge q,p\wedge q}f}^2_{\FH^{\otimes(2(l-l'))}},
		\end{aligned}
	\end{equation}
		where 
		\begin{equation*}
			\varphi_r=\sum_{i+j=r}\binom{p}{i}\binom{q}{i}\binom{q}{j}\binom{p}{j}i!j!f\tilde{\otimes}_{i,j} f.
		\end{equation*}
		Taking $c_1(p,q)=\min\Big\{\binom{p}{i}^2\binom{q}{j}^2 (p!q!)^2: 0<i+j<l\Big\}$, we get that
		\begin{align*}
			\E[\abs{F}^4]-2 \big(\E[\abs{F}^2]\big)^2  -\abs{\E[{F}^2]}^2\ge c_1(p,q)   \sum_{0<i+j<l}\norm{f {\otimes}_{i,j}h}^2_{\FH^{\otimes(2(l-i-j))}}.
		\end{align*}
	\end{proof}
	\begin{remark}
		Analogue to \eqref{revised version1}, one can show by using the similar argument that 
		\begin{equation}\label{revised version2}
			\begin{aligned}
				&\quad\E[\abs{F}^4]-2 \big(\E[\abs{F}^2]\big)^2 -\abs{\E[{F}^2]}^2\\
				&=\sum_{0 <i+j<l'}{p\choose i}{q\choose i}{q\choose j}{p\choose j} (p!q!)^2  \norm{f\otimes_{i,j}f}^2_{\FH^{\otimes(2(l-i-j))}} \\
				&\quad +\sum_{r=1}^{l-1}((l-r)!)^2\norm{\psi_r}^2_{\FH^{\otimes(2(l-r))}}\\
				&\quad+	\mathbf{1}_{\left\lbrace p\neq q\right\rbrace }\binom{p}{p\wedge q}^2\binom{q}{p\wedge q}^2(p! q! )^2\norm{f\otimes_{p\wedge q,p\wedge q}f}^2_{\FH^{\otimes(2(l-l'))}},
			\end{aligned}
		\end{equation}
		where $l'=2(p\wedge q)$ and
		\begin{equation*}
			\psi_r=\sum_{i+j=r}\binom{p}{i}^2\binom{q}{j}^2i!j!f\tilde{\otimes}_{i,j} h.
		\end{equation*}
		 Here we revise \cite[Equation 3.10, Equation 3.11]{cl19} as \eqref{revised version1} and \eqref{revised version2}. Note that the minor error of \cite[Equation 3.10, Equation 3.11]{cl19} dose not affect the main results in \cite{cl19}, since by \eqref{norm inequality 0}, if $p\neq q$, 
	$$\norm{f\otimes_{p\wedge q,p\wedge q}f}^2_{\FH^{\otimes(2(l-l'))}}\leq  \norm{f\otimes_{p-p\wedge q,q-p\wedge q}h}^2_{\FH^{\otimes 2l'}}\leq \sum_{0<i+j<l}\norm{f {\otimes}_{i,j}h}^2_{\FH^{\otimes(2(l-i-j))}}.$$
	\end{remark}

	\begin{lem}\label{3moment}
		If $F=I_{p,q}(f)$ with $f\in  \FH^{\odot p}\otimes \FH^{\odot q}$, then we have that
		\begin{equation}\label{3moment_1}
			\E\left[ F^3\right]=\mathbf{1}_{\left\lbrace p=q \right\rbrace }\sum_{i=0}^{p} i!(p-i)!(p!)^2\binom{p}{i}^4\left\langle f\tilde{\otimes}_{i,p-i}f,h \right\rangle _{\FH^{\otimes (2p)}},
		\end{equation}
		where $h$ is the reverse complex conjugate of $f$, and 
		\begin{equation}\label{3moment_2}
			\E\left[ F^2\bar{F}\right]=\mathbf{1}_{\left\lbrace p=q \right\rbrace }\sum_{i=0}^{p} i!(p-i)!(p!)^2\binom{p}{i}^4\left\langle f\tilde{\otimes}_{i,p-i}f,f \right\rangle _{\FH^{\otimes (2p)}}.
		\end{equation}
	\end{lem}
	
	\begin{proof}
		Applying product formula \eqref{Product_formula}, we have that
		\begin{equation*}
			F^2=\sum_{i=0}^{p\wedge q}\sum_{j=0}^{p\wedge q}i!j!\binom{p}{i}\binom{q}{i}\binom{p}{j}\binom{q}{j}I_{2p-(i+j),2q-(i+j)}\left( f\tilde{\otimes}_{i,j}f \right) 
		\end{equation*}
		Combining product formula \eqref{Product_formula} and isometry property \eqref{isometry property}, we obtain that 
		\begin{align*}
			\E\left[ F^3\right]&=\E\left[ F^2\bar{\bar{F}}\right]\\
			&=\sum_{i=0}^{p\wedge q}\sum_{j=0}^{p\wedge q}i!j!\binom{p}{i}\binom{q}{i}\binom{p}{j}\binom{q}{j}\E\left[ I_{2p-(i+j),2q-(i+j)}\left( f\tilde{\otimes}_{i,j}f \right) \overline{{I}_{q,p}(h)}\right] \\
			&=\sum_{i=0}^{p\wedge q}\sum_{j=0}^{p\wedge q}i!j!\binom{p}{i}\binom{q}{i}\binom{p}{j}\binom{q}{j}\mathbf{1}_{\left\lbrace 2p-(i+j)=q\right\rbrace }\mathbf{1}_{\left\lbrace 2q-(i+j)=p\right\rbrace } p!q!\left\langle  f\tilde{\otimes}_{i,j}f ,\right\rangle _{\FH^{\otimes (p+q)}}\\
			&=\mathbf{1}_{\left\lbrace p=q \right\rbrace }\sum_{i=0}^{p} i!(p-i)!(p!)^2\binom{p}{i}^4\left\langle f\tilde{\otimes}_{i,p-i}f,h \right\rangle _{\FH^{\otimes (2p)}}.
		\end{align*}
		Using the similar argument, we can get \eqref{3moment_2}. Then we finish the proof.
	\end{proof}
	
	\subsection{Proofs of the main results}\label{section 4.2}

	%--------------------------
	\noindent{\it Proof of Theorem~\ref{equilent cond}.\,}
	Let $l'=2(p\wedge q)$.
	It is known that \textup{(i)} is equivalent to that
	\begin{itemize}
		\item[\textup{(ii')}]  $\norm{f_n\ {{\otimes}}_{i,j} f_n}_{\FH^{\otimes ( 2(l-i-j))}}\to 0 $ for any $0<i+j\le l'-1$ and $\norm{f_n{\otimes}_{i,j} h_n}_{\FH^{\otimes ( 2(l-i-j))}}\to 0 $ for any $0<i+j\le l-1$,
	\end{itemize}please refer to \cite[Theorem 1.3]{chw17} or \cite[Theorem 3.14]{ cl19}. The equivalence between  \textup{(ii)} and  \textup{(ii')} is followed from Lemma~\ref{norm inequality lemma}.
	%--------------------------------------------------------------------------------------------------------
	{\hfill\large{$\Box$}}\\
	%--------------------------------------------------------------------------------------------------------

	%-------------------------

	%-------------------------------------------------------------------------------------------------------
	\noindent{\it Proof of Theorem~\ref{cFMBE bound}.\,}
	Denote $F=A+\mi B$. Then the decomposition theorem of \cite{cl17, cl19} implies that both $A$ and $B$  are $l$-th real Wiener-It\^{o} multiple integrals. That is to say, the 2-dimensional vector $(A,\,B)$ is a bivariate real Wiener-It\^{o} multiple integrals. 
	
	It follows from the fourth moment Berry-Ess\'{e}en bound (see Theorem 6.2.2 of \cite{NourPecc12} ) that 
	\begin{align}\label{bds 00}
		d_{W}(F, N)\le  4\sqrt{ \sum_{r=1}^{l-1} {2r \choose r}}  \frac{ \sqrt{\lambda_1}}{\lambda_2} (\abs{x_1}^{\frac12}+\abs{x_2}^{\frac12}),
	\end{align} where $x_1=\E[A^4]-3(\E[A^2])^2,\,x_2=\E[B^4]-3(\E[B^2])^2$.
	
	Since $x_1,\,x_2\ge 0$ (see Corollary 5.2.11 of \cite{NourPecc12}), we have that 
	\begin{align}\label{bds 01}
		\abs{x_1}^{\frac12}+\abs{x_2}^{\frac12}\le \sqrt{2(x_1+x_2)}.
	\end{align} Thus, 
\begin{equation}\label{bushengshi 1}
	\begin{aligned}
		x_1+x_2&= \E[A^4]-3(\E[A^2])^2+ \E[B^4]-3(\E[B^2])^2  \\
		&=\E[\abs{F}^4]-2 \big(\E[\abs{F}^2]\big)^2- \big|\E[F^2] \big|^2-2\left(\E\left[A^{2}B^{2} \right]-\E\left[A^{2}\right] \E\left[ B^{2}\right]-2\left(\E\left[ AB\right]  \right) ^2  \right) \\
		&\le \E[\abs{F}^4]-2 \big(\E[\abs{F}^2]\big)^2- \big| \E[F^2] \big|^2,
	\end{aligned}
\end{equation}
	where in the last line we use the fact (see \cite[Lemma 4.8]{cl17}):
	\begin{equation*}
		2\left(\E\left[A^{2}B^{2} \right]-\E\left[A^{2}\right] \E\left[ B^{2}\right]-2\left(\E\left[ AB\right]  \right) ^2\right) \geq0,
	\end{equation*}
	which can be proved by combining product formula \eqref{Product_formula}, isometry property \eqref{isometry property} and some combinatorics.
	
	Substituting the two inequalities  (\ref{bds 01}) and (\ref{bushengshi 1}) into (\ref{bds 00}), we obtain the desired Fourth moment Berry-Ess\'{e}en bound (\ref{Fourth moment BEB }). 
	%\subsection{ Fourth moment Berry-Ess\'{e}en bound for complex isotropic Gaussian random fields on $S^1$. }
	Together with Proposition~\ref{be bound prop 63}, the above Fourth Moment Berry-Ess\'{e}en bound gives the estimate \eqref{bound 2}.
	
	By \cite[Theorem 4.1]{c23}, we can further obtain that 
	\begin{align*}
		&d_{W}(F, N)\geq \sup\left\lbrace\left|\E\left[ g(\RE F,\IM F) \right] - \E\left[ g(\RE N,\IM N) \right]   \right|  \right\rbrace\\
		\geq& c(p,q,a,b,\sigma)\max\left\lbrace \left| \E\left[F^3 \right] \right|, \left|\E\left[F^2\bar{F} \right] \right|, \E\left[\left| F \right| ^4 \right]-2\left(\E\left[ \left| F \right| ^2\right]  \right) ^2-\left| \E\left[ F^2\right] \right| ^2 \right\rbrace,
	\end{align*}
	where $g:\mathbb{R}^2\rightarrow \mathbb{R}$ runs over the class of all four-times continuously differentiable functions such that $g$ and all of its derivatives of order up to four are bounded by one. Combining Proposition~\ref{be bound prop 63} and Lemma~\ref{3moment}, we obtain the lower bound \eqref{lower bound}.
	%--------------------------------------------------------------------------------------------------------
	{\hfill\large{$\Box$}}\\
	%--------------------------------------------------------------------------------------------------------
	
	%-------------------------------------------------------------------------------------------------------
	\noindent{\it Proof of Theorem~\ref{thm Camp}.\,} We denote by $d \times d$ matrices $\Sigma_{\RE}=(\Sigma_{\RE}(i,j))_{1\leq i,j\leq d}$ and $\Sigma_{\IM}=(\Sigma_{\IM}(i,j))_{1\leq i,j\leq d}$ the real and imaginary parts of $\Sigma$ respectively, namely, $\Sigma=\Sigma_{\RE}+\im \Sigma_{\IM}$.  Denote $F_{n}^{j}=A_{n}^{j}+\mi B_{n}^{j}$ for $1\leq j\leq d$. Since  $F_n $  is circularly symmetric, %invariant  by the action of $\mathbb{T}$, 
	we have $\E[F_n F_n^{'}] =0$. This together with \eqref{Fj bar Fk} ($\E[F_n \overline{F_n^{'}}] \rightarrow \Sigma$) implies that 
	\begin{align}
		\E\left[ A_n^iA_n^j\right]&=\E\left[ B_n^iB_n^j\right] \rightarrow \frac{1}{2}\Sigma_{\RE}(i,j),\\
		\E\left[ A_n^iB_n^j\right]&=-\E\left[ B_n^iA_n^j\right] \rightarrow -\frac{1}{2}\Sigma_{\IM}(i,j).
	\end{align}
	Then the assumption (6.2.8) of \cite[Theorem 6.2.3]{NourPecc12} is satisfied by the decomposition theorem of complex multiple Wiener-Ito integrals (see \cite[Theorem 3.3]{cl17}). Thus the equivalence between (i) and (ii) is obtained by \cite[Theorem 6.2.3]{NourPecc12}. Specially, from \cite[Section 3.7.6]{GRG13}, we know that $$F_n\overset{d}{\rightarrow}Z\sim\mathcal{CN}_d(0,\Sigma)$$ is equivalent to that
	\begin{equation}\label{vector convergence}
		\left(A_{n}^{1},\ldots,A_{n}^{d}, B_{n}^{1},\ldots,B_{n}^{d}\right) \overset{d}{\rightarrow}\mathcal{N}_{2d}(0,\Sigma'),
	\end{equation} where $\Sigma'$ is the $(2d)\times(2d)$ matrix defined as 
	$$
	\Sigma'=\begin{pmatrix}
		\frac{1}{2}\Sigma_{\RE} & -\frac{1}{2}\Sigma_{\IM} \\
		\frac{1}{2}\Sigma_{\IM} & \frac{1}{2}\Sigma_{\RE}
	\end{pmatrix}.
	$$
	
	By \cite[Theorem 6.2.3]{NourPecc12}, \eqref{vector convergence} is equivalent to that
	\begin{equation}\label{components convergence}
		A_{n}^{j}\overset{d}{\rightarrow}\mathcal{N}\left( 0,\frac{1}{2}\Sigma(j,j) \right),\quad B_{n}^{j}\overset{d}{\rightarrow}\mathcal{N}\left( 0,\frac{1}{2}\Sigma(j,j) \right), \quad 1\leq j\leq d.
	\end{equation} 
	Note that $\E\left[ A_n^jA_n^j\right]=\E\left[ B_n^jB_n^j\right]\rightarrow \frac{1}{2}\Sigma_{\RE}(j,j)$, $\E\left[ A_n^jB_n^j\right]\rightarrow -\frac{1}{2}\Sigma_{\IM}(j,j)=0$, using \cite[Theorem 6.2.3]{NourPecc12} again, we get that \eqref{components convergence} is equivalent that $$\left( A_{n}^{j}, B_{n}^{j}\right)\overset{d}{\rightarrow}\mathcal{N}_{2}\left( \begin{pmatrix}
		0\\0
	\end{pmatrix},\begin{pmatrix}
		\frac{1}{2}\Sigma(j,j) &0\\0 & \frac{1}{2}\Sigma(j,j)
	\end{pmatrix}\right), \quad 1\leq j\leq d,$$
	which is equivalent to that $F_n^j= A_{n}^{j}+\im B_{n}^{j}\overset{d}{\rightarrow}\mathcal{CN}\left( 0,\Sigma(j,j) \right) $ for $ 1\leq j\leq d$. Then we finish the proof. 
	%--------------------------------------------------------------------------------------------------------
	{\hfill\large{$\Box$}}\\
	%--------------------------------------------------------------------------------------------------------
	
	%-------------------------------------------------------------------------------------------------------
	\noindent{\it Proof of Theorem~\ref{thm Pecca TUd} .\,}
	By Theorem \ref{equilent cond}, conditions (i), (iii) and circular symmetry of $I_{p,q}\left( f_{n,p,q}\right)$ imply that for fixed $p,q$, $$I_{p,q}\left( f_{n,p,q}\right)\overset{d}{\rightarrow}\mathcal{CN}(0,\sigma^2_{p,q}), \quad n\rightarrow \infty . $$
	By Theorem \ref{thm Camp}, we get that the random vector composed by $\left\lbrace I_{p,q}( f_{n,p,q}): 1\leq p+q\leq M \right\rbrace $ converges in law to that composed by $\left\lbrace N_{p,q}:\,p+q\ge 1: 1\leq p+q\leq M \right\rbrace $, where $\set{N_{p,q}:\,p+q\ge 1}$ are the centered independent complex circularly-symmetric Gaussian random variables with variance $\sigma^2_{p,q}$. For every $M\ge 1$, set 
	$$	F_n^{(M)}=\sum_{p+q\ge 1}^{M}I_{p,q}( f_{n,p,q}),\quad N^{(M)}=\sum_{p+q\ge 1}^{M} N_{p,q},\quad N=\sum_{p+q\ge 1}N_{p,q}.$$ Then $F_n^{(M)}\overset{d}{\rightarrow}N^{(M)}$. Let $g:\R^2\rightarrow\R$  such that $\norm{g}_{\mathrm{Lip}}+\norm{g}_{\infty}\leq 1$, where $$\norm{g}_{\mathrm{Lip}}=\sup_{x\neq y, x,y\in\R^2}\frac{|g(x)-g(y)|}{\norm{x-y}_{\R^2}},\quad \norm{g}_{\mathrm{\infty}}=\sup_{x\in\R^2}|g(x)|.$$ Then for every $M\ge 1$, the Fortet-Mourier distance (or say, bounded Wassersterin distance, see \cite[p. 210]{NourPecc12})  between $F_n$ and $N$ satisfies that
	\begin{align*}
		& \quad d_{\mathrm{FM}}(F_n,\,N)  \\
		&=\sup_{g} \abs{\E\left[ g(\RE F_n,\IM F_n)\right] - \E\left[ g(\RE N,\IM N)\right] }  \\
		&\le \sup_{g} \Big[ \left| \E\left[ g(\RE F_n,\IM F_n)\right] - \E\left[ g\left( \RE F_n^{(M)},\IM F_n^{(M)}\right) \right]  \right|  +\abs{\E\left[ g\left( \RE F_n^{(M)},\IM F_n^{(M)}\right) \right]  \right. \\&\left. \quad\qquad-\E\left[ g\left( \RE F_n^{(M)},\IM F_n^{(M)}\right) \right]  } +\abs{\E\left[ g\left( \RE N,\IM N\right) \right]  -\E\left[ g\left( \RE N^{(M)}, \IM N^{(M)}\right) \right]  }\Big]  \\
		&\le \sqrt{\E\left[ \left| F_n-F_n^{(M)}\right| ^2\right] }+\sup_{g}\left|\E\left[ g\left( \RE F_n^{(M)},\IM F_n^{(M)}\right) \right] -\E\left[ g\left( \RE F_n^{(M)},\IM F_n^{(M)}\right) \right] \right|\\&\quad+\sqrt{\E\left[ \left| N-N^{(M)}\right| ^2\right] } \\
		&\le  \Big(  \sum\limits_{p+q> M}  p!q!\norm{f_{n,p,q}}^2_{\FH^{\otimes (p+q)}} \Big)^{\frac12}+\Big( \sum\limits_{p+q> M} \sigma^2_{p,q}\Big) ^{\frac12}\\&\quad + \sup_{g}\left|\E\left[ g\left( \RE F_n^{(M)},\IM F_n^{(M)}\right) \right] -\E\left[ g\left( \RE F_n^{(M)},\IM F_n^{(M)}\right) \right] \right|.
	\end{align*}
	Combining the fact that $F_n^{(M)}\overset{d}{\rightarrow}N^{(M)}$ 
	and the conditions (ii) and (iv), taking first the limit as $n\to \infty$ and then taking the limit as $M\to \infty$, we get that $d_{\mathrm{FM}}(F_n,\,N) \to 0 $ as $n\to\infty$. That is to say, $ F_n \stackrel{ {law}}{\to} \mathcal{CN}(0,\sigma^2)$.
	%--------------------------------------------------------------------------------------------------------
	{\hfill\large{$\Box$}}\\
	%--------------------------------------------------------------------------------------------------------

	%--------------------------------------------------------------------------------------------------------------%
	%-------------------------------------------------------------------------------------------------------
	\noindent{\it Proof of Theorem~\ref{duowei b-e bound thm}.\,}
	First, please refer to \cite[Theorem 3.7.12]{GRG13} for the relationship between the non-singularity and the eigenvalue, eigenvector pairs of $\Sigma$ and those of the real
	random $2d$-vector $(\RE Z, \IM Z)' $.
	
	Next, it follows from \cite[Theorem 4.3]{NR12} that when the pseudo-covariance matrix vanishes, i.e., $\E[F^2]=0$, we have that
	\begin{align*}
		d_{W}(F, Z)\le { \frac{2\sqrt{d \lambda_{max}}}{\lambda_{min}}}\sqrt{\E\norm{F_n}^4-\E\norm{N}^4}.
	\end{align*}
	It is well known that 
	\begin{align*}
		\E\norm{N}^4= \norm{\Sigma}^2_{F}+(\mathrm{Tr} \Sigma)^2,
	\end{align*}where $\norm{\Sigma}_{F}$ is the Hilbert-Schmidt (or say Frobenius)  norm. Hence, 
	\begin{align*}
		\E\norm{F}^4-\E\norm{N}^4=\sum_{j,r=1}^d \Big\{ \mathrm{Cov}(\abs{F^j}^2,\,\abs{F^r}^2) - \left|\E\left[ F^j\bar{F^r}\right]  \right| ^2\Big\}.
	\end{align*}
	Then the upper bound \eqref{key estimate 000} is an immediate consequence of Proposition~\ref{jproo key estimate} and the fact that $\E[F^jF^r]=0$ for $1\leq j,r\leq d$ which is from the circular symmetry of $F$.
	%--------------------------------------------------------------------------------------------------------
	{\hfill\large{$\Box$}}\\
	%--------------------------------------------------------------------------------------------------------
	
	%-------------------------------------------------------------------------------------------------------
	
	\bibliographystyle{abbrv}
	
	\bibliography{refs}

	\begin{comment}
	
\end{comment}	
	%-------------------------------------------------------------------------------------------------------
\end{document}